\documentclass[11pt]{amsart}
\usepackage[a4paper,margin=2.35cm]{geometry}
\usepackage{amsmath,amssymb,amsthm,mathtools}
\usepackage{enumitem}
\usepackage{microtype}
\usepackage{xcolor}
\usepackage[colorlinks=true,linkcolor=blue!55!black,citecolor=blue!55!black,urlcolor=blue!55!black]{hyperref}

\newcommand{\dd}{\,\mathrm d}
\newcommand{\spt}{\operatorname{spt}}

\theoremstyle{plain}
\newtheorem{theorem}{Theorem}[section]
\newtheorem{proposition}[theorem]{Proposition}
\newtheorem{lemma}[theorem]{Lemma}
\theoremstyle{remark}

\numberwithin{equation}{section}

\title{Global \(W^{2,p}\) Regularity in Optimal Transport}

\author{Shibing Chen}
\address{School of Mathematical Sciences,
University of Science and Technology of China,
Hefei, Anhui 230026, China}
\email{chenshib@ustc.edu.cn}

\author{Yuanyuan Li}
\address{Institute for Theoretical Sciences,
Westlake University, Hangzhou, 310030, China}
\email{lyyuan@westlake.edu.cn}

\author{Xianduo Wang}
\address{School of Mathematics and Statistics,
Huazhong University of Science and Technology,
Wuhan, 430074, China}
\email{wangxd@hust.edu.cn}
\date{}

\begin{document}

\begin{abstract}
In this paper we establish global \(W^{2,p}\) estimates for the convex potentials of quadratic optimal transport between bounded convex domains with continuous positive densities. All the assumptions are optimal. The main new ideas include a blow-up analysis that allows for lower-dimensional collapse of the limiting source measure and reduces the limiting problem to a transport problem on its affine hull, and a good--bad scale decomposition in which rigidity controls the good scales while a counting argument shows that the proportion of bad scales tends to zero.
\end{abstract}

\maketitle

\section{Introduction}

Let \(\Omega,\Omega^*\subset\mathbb R^n\) be bounded open convex sets, and assume that
\begin{equation}\label{eqdata}
 f\in C(\overline\Omega),\qquad
 g\in C(\overline{\Omega^*}),\qquad
 0<\lambda\le f,g\le\Lambda,\qquad
 \int_\Omega f=\int_{\Omega^*}g.
\end{equation}
Let \(u\) and \(v\) be the dual convex potentials for the quadratic optimal transport, so that
\begin{equation}\label{breniereq}
 (Du)_\#(f\,\dd x)=g\,\dd y,
 \qquad
 (Dv)_\#(g\,\dd y)=f\,\dd x.
\end{equation}
 As usual, the potentials are extended to \(\mathbb R^n\) by the minimal convex extension described in Section~\ref{section2}.

The existence and uniqueness of the optimal transport map were established by Brenier \cite{Brenier1991}; see also McCann \cite{McCann1995} for the more general theory of monotone maps. For the interior regularity theory of the Monge--Amp\`ere equation, we refer the reader to the fundamental works \cite{Caffarelli1990W2p, Caffarelli1991, Caffarelli1992Mappings, PhillipisFigalli2013,PhillipisFigalliSavin2013}.
For general background on optimal transport and the Monge--Amp\`ere equation, we refer to the monographs \cite{Villani2003,Villani2009,Figalli2017}. The fundamental regularity condition for optimal transport with general cost functions was introduced by Ma--Trudinger--Wang \cite{MaTrudingerWang2005}, and was subsequently shown by Loeper \cite{Loeper2009} to be necessary for the continuity of optimal transport maps for arbitrary smooth positive densities. For global Sobolev regularity of optimal transport potentials on a class of nonconvex planar domains, see
\cite{MooneyRakshit2024, ChenLiLiu2023}.

Under stronger smoothness assumptions, Delano\"e \cite{Dela} first established global classical solvability in dimension two for uniformly convex domains. Caffarelli \cite{Caffarelli1992Boundary} subsequently established global \(C^{1,\alpha}\) estimates for bounded convex domains when the densities are 
bounded above and below. Later, in \cite{Caffarelli1996Boundary}, he proved global \(C^{1,1-\varepsilon}\) estimates when the domains are uniformly convex and \(C^2\) and the densities are positive and continuous. In the same paper, 
Caffarelli also established global \(C^{2,\alpha'}\) estimates, assuming further that the densities are \(C^\alpha\), for some \(\alpha'<\alpha\). Urbas \cite{Urbas}  extended the classical solvability result of Delano\"e to higher dimensions. 

Chen--Liu--Wang \cite{ChenLiuWang2021} proved global \(C^{1,1-\varepsilon}\) and \(W^{2,p}\) estimates for continuous positive densities when the convex domains have \(C^{1,1}\) boundaries, and established \(C^{2,\alpha}\) estimates when the domains are \(C^{1,1}\) and the densities are \(C^\alpha\). Savin--Yu obtained global \(C^{1,1-\varepsilon}\) and \(W^{2,p}\) estimates on arbitrary planar convex domains with constant densities \cite{SavinYu2020}. 
By introducing a very useful monotonicity formula, a related version of which was also used by Yuan \cite{Yuan2023} to give a new proof of the Pogorelov estimate for the Monge--Amp\`ere equation, Collins--Tong extended the results of Savin--Yu to all dimensions under a H\"older continuity assumption on the densities \cite{CollinsTong2025}. They also established global \(C^{2,\alpha}\) estimates under the sharp assumptions, namely, that the domains are convex and \(C^{1,\alpha}\) and that the densities are \(C^\alpha\). However, \(W^{2,p}\) estimates under the sharp assumptions have remained open. In this paper, we give a complete answer.

\begin{theorem}\label{mainthm}
Suppose \eqref{eqdata} and \eqref{breniereq} hold. Then, for every \(p>1\),
\[
 u\in W^{2,p}(\overline \Omega),
 \qquad
 v\in W^{2,p}(\overline{\Omega^*}).
\]
\end{theorem}

There are three main difficulties. First, near the boundary, sections may be cut by the domains in an irregular way, since the domains are not assumed to be smooth or strictly convex. Second, continuity of the densities only implies that their oscillations tend to zero, which does not imply that the errors over all scales are summable. Third, after affine normalization, the limiting source measure may be supported on a lower dimensional subspace.

The proof does not require summability of the density errors: their average over dyadic scales tends to zero, so only $o(N)$ scales are bad. The blow-up analysis also allows the limiting source measure to collapse onto a proper subspace.

The proof consists of the following five steps: 
\begin{itemize}[leftmargin=2.2em]
\item[(i)] \emph{An almost-monotonicity estimate.}
Fix \(x_0\in\overline\Omega\), set \(y_0=Du(x_0)\), and define
\[
 D_r(u,x_0)
 :=\bigl\{x\in\Omega:
 (x-x_0)\cdot(Du(x)-y_0)<r^2\bigr\},
 \qquad
 F_{x_0}(r)
 :=r^{-n}\int_{D_r(u,x_0)}f(x)\,\dd x .
\]
Adapting the monotonicity formula and first-variation argument of Collins--Tong
\cite[Theorem~3.1 and Proposition~3.1]{CollinsTong2025},
we prove that
\[
 F_{x_0}(r_1)
 \le F_{x_0}(r_2)
 \exp\!\left(
 C\int_{r_2}^{r_1}
 \frac{\omega_f(Cs^\sigma)+\omega_g(Cs^\sigma)}{s}\,\dd s
 \right)
\]
whenever \(0<r_2<r_1\le r_0\). The error on the right is additive on adjacent scale intervals. Although it need not be summable down to zero, its average over dyadic logarithmic scales tends to zero. Note that no Dini condition is required.

\item[(ii)] \emph{Blow-up analysis.}
We perform the normalization in the same spirit as in \cite{ChenLiuWang2024}. More precisely, the normalization is applied to sections of the source potential, in contrast to \cite{ChenLiuWang2024}, where it is applied only on the target side. We normalize centered sections using their John ellipsoids. The geometric decay of the centered sections controls the normalized potentials and yields a locally uniform limit that is finite on all of \(\mathbb R^n\). At the same time, we consider the limits of the corresponding dual potentials. This allows us to handle the possible degeneration in which the limiting source measure is supported on a proper linear subspace.

\item[(iii)] \emph{The limit and dimensional reduction.}
In the full-dimensional case, we follow \cite[Theorem~4.1]{CollinsTong2025}. If the limiting source measure is supported on an \(m\)-dimensional subspace \(L\), with \(q=n-m\), projection onto \(L\) identifies the limiting density with the normalized \(q\)-dimensional volumes of transverse sections. By convexity and the Brunn--Minkowski inequality, its \(q\)-th root is concave, and the projected measures have uniformly bounded densities. The same argument applies to the projected target measure. The mass identity then yields a two-homogeneous \(m\)-dimensional limiting transport problem, which can be lifted back to the original variables.

\item[(iv)] \emph{Rigidity and iteration.}
The rigidity of the blow-up limits implies that, on a sufficiently long interval of scales with small total error, the next sub-level set differs little from one half of the preceding one. The sum of the errors over the first \(N\) dyadic scales is \(o(N)\), so the estimate fails at only \(o(N)\) indices. Consequently, for every \(\delta>0\),
\[
 B_{c_\delta h^{1/2+\delta}}(x)
 \subset S_h^c[u](x)
 \subset B_{C_\delta h^{1/2-\delta}}(x).
\]

\item[(v)] \emph{Conclusion.}
Finally, the \(W^{2,p}\) estimate follows from combining Caffarelli's interior estimate with a  covering argument of Savin \cite{SavinGlobalW2p}.
\end{itemize}

We would also like to point out that the methods developed in this paper can be used to slightly improve the results in \cite{ChenLiu2025}. More precisely, one can relax the regularity assumptions on the target domains by removing the uniform convexity condition and requiring only \(C^{1,\alpha}\) regularity (resp. mere convexity) to obtain \(C^{2,\alpha}\) (resp. \(C^{1,1-\epsilon}\)) regularity of the singular set.

The paper is organized as follows. Section~\ref{section2} contains the
notation and the standard estimates for sub-level sets. In Section~\ref{secendpoint} we prove the mass estimate. Sections \ref{secblowup} and \ref{sec5limit} contain the blow-up analysis and dimension reduction arguments. The limit profile is identified in Section~\ref{sec6rigidity}, and the iteration and the proof of the main theorem are carried out in Section~\ref{sec7shape}. Appendix~\ref{appendix} contains the local regularity result used for locally finite measures.

\section{Preliminaries}
\label{section2}

Throughout the paper, \(c,C>0\) denote universal constants depending only on the fixed data. Note that their values may change from line to line.

Let \(u\) and \(v\) be the potentials satisfying \eqref{breniereq}. For a set \(D\) and a convex function \(\phi\), define
\[
 I_D(z)=
 \begin{cases}
  0,&z\in D,\\
  +\infty,&z\notin D,
 \end{cases}
 \qquad
 \phi^*(z)=\sup_{\xi\in\mathbb{R}^n}\{z\cdot\xi-\phi(\xi)\}.
\]
Then we extend \(u,v\) to convex functions on \(\mathbb R^n\) by
\begin{equation}\label{uvextensions}
 \overline u=(v+I_{\overline{\Omega^*}})^*,
 \qquad
 \overline v=(u+I_{\overline\Omega})^*.
\end{equation}
We call \(\overline u,\overline v\) the minimal convex extensions of \(u,v\). They are finite on \(\mathbb R^n\), agree with \(u,v\) on their respective closed domains, and satisfy
\begin{equation}\label{uvdual}
 \overline u^*=v+I_{\overline{\Omega^*}},
 \qquad
 \overline v^*=u+I_{\overline\Omega}.
\end{equation}
Note that \((x,y)\in\overline\Omega\times\overline{\Omega^*}\) satisfies
\[
 u(x)+v(y)=x\cdot y
\]
if and only if
\(y\in\partial\overline u(x)\) and
\(x\in\partial\overline v(y)\).

For \(x_0\in\overline\Omega\), let
\[
\ell_{x_0}(x)= u(x_0)+D u(x_0)\cdot(x-x_0).
\]
For \(h>0\), we define
\[
S_h[u](x_0)=\{x\in\Omega:u(x)<\ell_{x_0}(x)+h\}.
\]
The centered sub-level set of height \(h\) at \(x_0\) is
\begin{equation}\label{centeredsec}
	S_h^c[u](x_0)
	=\{x\in\mathbb R^n:\overline u(x)<\widehat\ell(x)+h\},
\end{equation}
where \(\widehat\ell\) is affine, \(\widehat\ell(x_0)=u(x_0)\), and is chosen so that \(x_0\) is the center of mass of \(S_h^c[u](x_0)\). The existence of \(\widehat\ell\) follows from \cite[Section~2]{Caffarelli1996Boundary}. We also set
\[
D_r(u,x_0)=\bigl\{x\in\Omega:
(x-x_0)\cdot\bigl(Du(x)-Du(x_0)\bigr)<r^2\bigr\},
\qquad r>0.
\]
Note that $S_h[u](x_0)$ and $D_r(u,x_0)$ are contained in $\Omega,$ but $S_h^c[u](x_0)$ may contain both points in and out of $\Omega$. When no confusion arises, we may abbreviate them as $S_h[u], S_h^c[u]$ or $S_h(x_0), S_h^c(x_0)$. The corresponding sets for \(v\) are defined by interchanging \((\Omega,u)\) and \((\Omega^*,v)\).

We first recall the estimates used throughout the paper. 
\begin{proposition}
\label{classicalproperties}
There are \(\alpha_0\in(0,1)\) and constants depending only on \(n,\lambda,\Lambda\) and the fixed inner and outer radii of the two domains such that the following hold.
\begin{enumerate}[label=\textup{(\roman*)}]
\item By \cite[Theorem~1.1]{JhaveriSavin2022} and \cite[Proposition 2.5]{CollinsTong2025}, we have
\[
 [D u]_{C^{0,\alpha_0}(\overline\Omega)}
 +[D v]_{C^{0,\alpha_0}(\overline{\Omega^*})}\le C.
\]
For \(x_0\in\overline\Omega\), \(y_0\in\overline{\Omega^*}\), and every fixed \(R<\infty\),
\begin{equation}\label{globalupper}
    [D\overline u]_{C^{0,\alpha_0}(B_R(x_0))}
     +[D\overline v]_{C^{0,\alpha_0}(B_R(y_0))}\le C_R.
\end{equation}

\item By \cite{Caffarelli1992Boundary}, we know the restrictions
\[
 Du:\Omega\longrightarrow\Omega^*,
 \qquad
 Dv:\Omega^*\longrightarrow\Omega
\]
are mutually inverse homeomorphisms.

\item By \cite[Proposition~2.6]{CollinsTong2025} and \cite[Proposition~2.3(3)]{CollinsTong2025}, for every \(x_0\in\overline\Omega\), \(r>0\), and \(h>0\),
\begin{align}
 x_0+\tfrac12\bigl(S_{r^2}[u](x_0)-x_0\bigr)
 &\subset D_r(u,x_0)
 \subset S_{r^2}[u](x_0),
 \label{include-D-S}\\
 S_{ch}^c[u](x_0)\cap\Omega
 &\subset S_h[u](x_0)
 \subset S_{Ch}^c[u](x_0)\cap\Omega.
 \label{include-S-Sc}
\end{align}
The analogous inclusions hold for \(v\).

\item By \cite[Proposition~2.3(1),(2)]{CollinsTong2025}, let \(p_h=D\widehat\ell\), where \(\widehat\ell\) defines \(S_h^c[u](x_0)\) in \eqref{centeredsec}, and let \(A_h\) be the positive definite linear map which normalizes the John ellipsoid of \(S_h^c[u](x_0)-x_0\) onto \(B_1\). Then
		\begin{align}
			B_c
			&\subset h^{-1}A_h^{-t}
			\bigl(Du(S_h^c[u](x_0))-p_h\bigr)
			\subset B_C,
			\label{gradpolar}\\
			c h^n
			&\le
			|S_h^c[u](x_0)|\,
			\bigl|Du(S_h^c[u](x_0))\bigr|
			\le C h^n.
			\label{volumeproduct}
		\end{align}
The same estimates hold with \(u\) replaced by \(v\).

\item For every \(0<h_-\le h_+<\infty\), there are \(0<c_{h_-,h_+}\le C_{h_-,h_+}<\infty\) such that, for every \(x_0\in\overline\Omega\), \(y_0=D\overline u(x_0)\), and
\(h\in[h_-,h_+]\),
\begin{align*}
 B_{c_{h_-,h_+}}(x_0)
 &\subset S_h^c[u](x_0)
 \subset B_{C_{h_-,h_+}}(x_0),\\
 B_{c_{h_-,h_+}}(y_0)
 &\subset S_h^c[v](y_0)
 \subset B_{C_{h_-,h_+}}(y_0).
\end{align*}
\end{enumerate}
\end{proposition}

\begin{proof}

We only need to prove \textup{(v)} and it suffices to give the proof for \(u\). Suppose that the outer inclusion fails. Then there are \(x_k\in\overline\Omega\), 
\(h_k\in[h_-,h_+]\), and centered sections \(S_{h_k}^c[u](x_k)\) whose outer radii tend to infinity. After passing to a subsequence,
\[
 x_k\to x_\infty,\qquad D\overline u(x_k)\to D\overline u(x_\infty),\qquad h_k\to h_\infty,
\]
and the convex functions
\[
 z\longmapsto \overline u(x_k+z)-\overline u(x_k)-D\overline u(x_k)\cdot z
\]
converge locally uniformly to a convex function \(\phi_\infty\). Let \(e_k\) be a direction of maximal radius and assume that \(e_k\to e\). Both radial endpoints on \(x_k+\mathbb Re_k\) tend to infinity. Thus, for every fixed \(t>0\), the points \(x_k \pm te_k\) belong to $S_{h_k}^c$ for all large \(k\). Adding the two defining inequalities and passing to the limit, we obtain
\[
 0\le\phi_\infty(te)+\phi_\infty(-te)\le2h_+.
\]
The function $\phi_\infty(te)+\phi_\infty(-te)$ is convex in \(t\), vanishes at \(0\), and is bounded on \([0,\infty)\). Therefore, it is identically zero. Thus \(u\) is affine on the line \(x_\infty+\mathbb Re\), which contradicts the fact that its subgradient image has nonempty interior. This proves the uniform outer inclusion.

For the inner inclusion, let \(p_h=D\widehat\ell_h\). The existence of a bounded section of \(\overline u-p_h\cdot(\,\cdot-x_0)\) implies
\[
 p_h\in\operatorname{dom}\overline u^*=\overline{\Omega^*}.
\]
Set \(L_0=1+\sup_{y\in\overline{\Omega^*}}|y|\). Since \(D\overline u(\mathbb R^n)\subset\overline{\Omega^*}\), both \(\overline u\) and the affine function with slope \(p_h\) are \(L_0\)-Lipschitz. Therefore,
\[
 \overline u(x_0+z)-\overline u(x_0)-p_h\cdot z\le2L_0|z|<h_-\le h
\]
whenever \(|z|<h_-/(2L_0)\) and the argument for \(v\) is the same.
\end{proof}

We shall also use the following two elementary consequences.
\begin{lemma}\label{lemslope}
Let \(\phi:\mathbb R^n\to\mathbb R\) be finite and convex, with \(\phi(0)=0\) and \(0\in\partial\phi(0)\). If
\[
 B_r\subset S_h^c[\phi](0)=\{x\in\mathbb R^n:\phi(x)-p_h\cdot x<h\},
\]
then \(|p_h|\le h/r\).
\end{lemma}

\begin{proof}
Since \(0\in\partial\phi(0)\) and \(\phi(0)=0\), we have \(\phi\ge0\). If \(p_h\ne0\), set \(e=p_h/|p_h|\). For \(0<\varepsilon<r\), the point \(-(r-\varepsilon)e\) belongs to \(S_h^c[\phi](0)\), and hence
\[
 0\le\phi(-(r-\varepsilon)e)<h-(r-\varepsilon)|p_h|.
\]
Letting \(\varepsilon\rightarrow0\) proves the assertion. The case \(p_h=0\) is immediate.
\end{proof}

The next lemma is standard and it compares centered sections at two heights.

\begin{lemma}\label{fixratio}
Let \(\phi:\mathbb R^n\to\mathbb R\) be finite and convex, and suppose that $S_h^c[\phi](0)$ is bounded. Then, for every \(0<\theta\le1\),
\begin{equation}\label{fixedratio}
 \frac{\theta}{n^3}S_h^c[\phi](0)\subset S_{\theta h}^c[ \phi](0)\subset n^3S_h^c[\phi](0).
\end{equation}
\end{lemma}

\begin{proof}
After subtracting an affine function, \(\phi(0)=0\) and \(\phi\ge0\). Fix \(e\in\mathbb S^{n-1}\), and write the two boundary points of \(S_h^c\cap\mathbb Re\) as \(a_he\) and \(-b_he\), where \(a_h,b_h>0\). Set \(R_h=\max\{a_h,b_h\}\). Since $S_h^c$ is centered at the origin, by John's lemma \cite[Lemma~A.3]{CaffarelliMcCann2010}, with \(\gamma_n=n^{3/2}\), we have
\begin{equation}\label{oppradii}
 \frac{R_h}{\gamma_n}\le a_h,b_h\le R_h.
\end{equation}
Since $\frac{\phi(a_he)}{a_h}+\frac{\phi(-b_he)}{b_h}=h\left(\frac1{a_h}+\frac1{b_h}\right),$ and \(t\mapsto\phi(\pm te)/t\) is nondecreasing on \((0,\infty)\), the even convex function \(G_e(t)=\phi(te)+\phi(-te)\) satisfies
\begin{equation}\label{raheibound}
 G_e(R_h)\ge2h,\qquad G_e(R_h/\gamma_n)\le2h.
\end{equation}
Let \(R_{\theta h}\) be defined in the same way at height \(\theta h\). If \(R_{\theta h}<R_h/\gamma_n\), the monotonicity of \(G_e(t)/t\) and \eqref{raheibound} imply
\[
 \frac{2\theta h}{R_{\theta h}}\le\frac{2\gamma_n h}{R_h}.
\]
Thus \(R_{\theta h}\ge\theta R_h/\gamma_n\). The same conclusion is immediate when \(R_{\theta h}\ge R_h/\gamma_n\). If \(0<\theta<1\), then
\[
 G_e(R_{\theta h}/\gamma_n)\le2\theta h<2h\le G_e(R_h),
\]
so \(R_{\theta h}\le \gamma_nR_h\). This inequality is automatic when \(\theta=1\). Combining these bounds with \eqref{oppradii} at the two heights, we obtain
\[
 \frac{\theta}{\gamma_n^2}a_h\le a_{\theta h}\le\gamma_n^2a_h
\]
in every direction. Since \(\gamma_n^2=n^3\), \eqref{fixedratio} follows.
\end{proof}

\section{An almost monotonicity estimate}
\label{secendpoint}

Fix \(x_0\in\overline\Omega\) and set \(y_0=D\overline u(x_0)\). Define
\[
 \Psi_{x_0}(x)
 :=(x-x_0)\cdot\bigl(Du(x)-y_0\bigr),
 \qquad
 \Psi_{y_0}^*(y)
 :=(y-y_0)\cdot\bigl(Dv(y)-x_0\bigr),
\]
and, for \(h>0\),
\[
 D_{\sqrt h}(u,x_0) =\{x\in\Omega:\Psi_{x_0}(x)<h\},
 \qquad
 D_{\sqrt h}(v,y_0)=\{y\in\Omega^*:\Psi_{y_0}^*(y)<h\}.
\]
When \(y=Du(x)\), the Legendre identities imply
\(\Psi_{y_0}^*(y)=\Psi_{x_0}(x)\). Hence
\begin{equation}\label{transportedmass}
 M_{x_0}(h)
 :=\int_{D_{\sqrt h}(u,x_0)}f(x)\,\dd x
 =\int_{D_{\sqrt h}(v,y_0)}g(y)\,\dd y.
\end{equation}

Define the two global moduli
\begin{align*}
 \omega_f(t)&=\sup\{|f(x)-f(x')|:
 x,x'\in\overline\Omega,\ |x-x'|\le t\},\\
 \omega_g(t)&=\sup\{|g(y)-g(y')|:
 y,y'\in\overline{\Omega^*},\ |y-y'|\le t\}.
\end{align*}

By \eqref{globalupper}, for every fixed \(R<\infty\) and $|z|,|\zeta|\le R$,
\begin{align*}
 0&\le
 \overline u(x_0+z)- \overline u(x_0)-y_0\cdot z
 \le C_R|z|^{1+\alpha_0}
 ,\\
 0&\le
 \overline v(y_0+\zeta)- \overline v(y_0)-x_0\cdot\zeta
 \le C_R|\zeta|^{1+\alpha_0}.
\end{align*}

It follows from \eqref{uvextensions} and the above estimates that
\begin{equation*}
 \begin{aligned}
  u(x)- u(x_0)-y_0\cdot(x-x_0)
 &\ge c|x-x_0|^{1+1/\alpha_0},\\
  v(y)- v(y_0)-x_0\cdot(y-y_0)
 &\ge c|y-y_0|^{1+1/\alpha_0},
 \end{aligned}
\end{equation*}
for \(x\in\overline\Omega\), \(y\in\overline{\Omega^*}\), uniformly in
\(x_0,y_0\).

If \(y=Du(x)\), a direct computation shows that
\begin{equation*}
 \Psi_{x_0}(x)=
  u(x)- u(x_0)-y_0\cdot(x-x_0)
 + v(y)- v(y_0)-x_0\cdot(y-y_0).
\end{equation*}
Both terms on the right are nonnegative. Consequently, there are \(C_0,r_0>0\), depending only on the fixed data, such that
\begin{equation}\label{diameterup}
 \operatorname{diam}D_r(u,x_0)
 +\operatorname{diam}D_r(v,y_0)
 \le C_0r^\sigma,
 \qquad
 \sigma=\frac{2\alpha_0}{1+\alpha_0},
 \qquad 0<r\le r_0.
\end{equation}
Define
\begin{equation*}
 \vartheta(r)=\omega_f(C_0r^\sigma)+\omega_g(C_0r^\sigma),
 \qquad
 F_{x_0}(r)=r^{-n}M_{x_0}(r^2).
\end{equation*}

For homogeneous densities, Collins--Tong established an exact monotonicity formula \cite[Theorem~3.1]{CollinsTong2025}. Their smooth computation \cite[Proposition~3.1]{CollinsTong2025} also identifies the error terms associated with nonhomogeneous H\"older densities. We give the argument because the densities are only continuous here and we estimate these terms directly by \(\omega_f\) and \(\omega_g\).

\begin{theorem}
\label{nearlymonotonicity}
There is a constant $C$, depending only on $n,\lambda,\Lambda$, such that
\begin{equation}\label{endpointinequality}
 F_{x_0}(r_1)
 \le F_{x_0}(r_2)
 \exp\!\left(C\int_{r_2}^{r_1}\frac{\vartheta(s)}s\,\dd s\right)
\end{equation}
whenever $0<r_2<r_1\le r_0$. 
\end{theorem}

\begin{proof}
After a change of coordinates and subtraction of an affine function, we may assume that $x_0=y_0=0$ and 
\[
 u(0)=v(0)=0,\qquad Du(0)=Dv(0)=0,\qquad \Psi(x)=x\cdot Du(x),\qquad \Psi^*(y)=y\cdot Dv(y).
\]

\smallskip
\noindent\emph{Step 1: the smooth case.}
Assume first that the domains are smooth and uniformly convex and that the data and potentials are smooth. These assumptions are used only in the coarea formula, the divergence theorem, and the classical change of variables. The resulting estimate is independent of all higher-order norms. Set \(T=Du\). A direct computation shows
\[
 D\Psi=T+D^2u(x)x,\qquad D\Psi^*(T(x))=x+\left(D^2u(x)\right)^{-1}T(x).
\]
For almost every \(h>0\), we have
\begin{equation}\label{diffM}
 M'(h)=\int_{\Omega\cap\{\Psi=h\}}\frac{f}{|D\Psi|}\,\dd\mathcal H^{n-1},
\end{equation}
where \(M(h)=M_{x_0}(h)\). Define
\[
 \begin{aligned}
 A_f(h)&=\int_{\Omega\cap\{\Psi=h\}}f\frac{x\cdot D\Psi}{|D\Psi|}\,\dd\mathcal H^{n-1},\\
 I_f(h)&=\int_{\Omega\cap\{\Psi<h\}}x\cdot Df\,\dd x,\\
 \mathcal B_f(h)&=\int_{\partial\Omega\cap\{\Psi<h\}}f\,x\cdot\nu_\Omega\,\dd\mathcal H^{n-1},
 \end{aligned}
\]
and define \(A_g,I_g,\mathcal B_g\) analogously on the target. The divergence theorem gives
\begin{equation}\label{stdivergence}
 A_f+\mathcal B_f=nM+I_f,\qquad A_g+\mathcal B_g=nM+I_g.
\end{equation}
A direct computation shows that 
$$d\mathcal{H}^{n-1}_y
=
\det D^2u\,
\frac{|D\Psi^*(T(x))|}{|D\Psi(x)|}
\,d\mathcal{H}^{n-1}_x.$$
Pulling the target coarea formula back by \(T\), as in \cite[proof of Proposition~3.1]{CollinsTong2025}, we obtain
\begin{equation}\label{stsum}
 A_f(h)+A_g(h)=4hM'(h)+\mathcal E(h),
\end{equation}
where
\[
 \mathcal E(h)=\int_{\Omega\cap\{\Psi=h\}}\frac{f}{|D\Psi|}\left|\left(D^2u(x)\right)^{1/2}x-\left(D^2u(x)\right)^{-1/2}T(x)\right|^2\,\dd\mathcal H^{n-1}\ge0.
\]
Combining \eqref{stdivergence} and \eqref{stsum}, we find
\begin{equation}\label{exactidentity}
 4hM'(h)-2nM(h)=I_f(h)+I_g(h)-\mathcal B_f(h)-\mathcal B_g(h)-\mathcal E(h).
\end{equation}
Since $0\in \overline{\Omega}\cap \overline{\Omega^*}$, by convexity of the two domains, we have 
\[
 x\cdot\nu_\Omega(x)\ge0,\qquad y\cdot\nu_{\Omega^*}(y)\ge0.
\]
Consequently, \(\mathcal B_f,\mathcal B_g\ge0\).

It remains to estimate \(I_f\) and \(I_g\). On every admissible ray \(x=t\theta\), the function
\[
 p_\theta(t)=\theta\cdot Du(t\theta)
\]
is nonnegative and nondecreasing. Thus \(D_{\sqrt h}(u,0)\) is star-shaped about the origin. If \(R_h(\theta)\) is its radial endpoint, integration by parts on each ray gives
\[
 I_f(h)=n\int_{\mathbb S^{n-1}}\int_0^{R_h(\theta)}
 \bigl[f(R_h(\theta)\theta)-f(t\theta)\bigr]
 t^{n-1}\,\dd t\,\dd\theta.
\]
Therefore
\begin{equation}\label{densityerrorbound}
\begin{aligned}
 (I_f(h))_+&\le\frac n\lambda\omega_f\!\left(\operatorname{diam}D_{\sqrt h}(u,0)\right)M(h),\\
 (I_g(h))_+&\le\frac n\lambda\omega_g\!\left(\operatorname{diam}D_{\sqrt h}(v,0)\right)M(h).
\end{aligned}
\end{equation}
By \eqref{densityerrorbound}, \eqref{exactidentity} and \eqref{diameterup}, we have
\[
 \frac{\dd}{\dd h}\bigl(h^{-n/2}M(h)\bigr)\le C\frac{\vartheta(\sqrt h)}h h^{-n/2}M(h)
\]
for almost every \(h\in(0,r_0^2]\). Since \(F_{x_0}(r)=r^{-n}M(r^2)\), it follows 
\begin{equation*}
 F_{x_0}'(r)\le C\frac{\vartheta(r)}rF_{x_0}(r)
 \quad\text{for a.e. }r\in(0,r_0].
\end{equation*}
in the smooth setting. Integration proves \eqref{endpointinequality}.

\smallskip
\noindent\emph{Step 2: approximation.}
Let \(P_{\overline\Omega},P_{\overline{\Omega^*}}\) be the metric projections onto the closed convex domains, namely, 
$$P_{\overline\Omega}(x)=\operatorname*{argmin}_{z\in \overline\Omega}|x-z|, \ \ P_{\overline{\Omega^*}}(x)=\operatorname*{argmin}_{z\in \overline{\Omega^*}}|x-z|.$$
Choose smooth uniformly convex outer approximations
\[
 \Omega_j\rightarrow\overline\Omega,\qquad
 \Omega_j^*\rightarrow\overline{\Omega^*},
\]
and for \(\varepsilon_j\rightarrow0\), define
\[
 f_j=(f\circ P_{\overline\Omega})*\rho_{\varepsilon_j},\qquad g_j=\kappa_j\bigl((g\circ P_{\overline{\Omega^*}})*\rho_{\varepsilon_j}\bigr),\qquad \kappa_j=\frac{\int_{\Omega_j}f_j\,\dd x}{\int_{\Omega_j^*}\bigl((g\circ P_{\overline{\Omega^*}})*\rho_{\varepsilon_j}\bigr)\,\dd y}.
\]
The metric projections are one-Lipschitz, convolution does not increase a modulus of continuity, and the outer approximations converge in measure. Hence \(\kappa_j\to1\), and, for all sufficiently large \(j\),
\begin{equation}\label{approximationdensity}
 \frac\lambda2\le f_j,g_j\le2\Lambda,\qquad \omega_{f_j}\le\omega_f,\qquad \omega_{g_j}\le2\omega_g.
\end{equation}

Let \(u_j\) be the corresponding transport potentials, extended to \(\mathbb R^n\) as in Section~\ref{section2}. After the normalization, the approximating domains have uniform inner and outer radii, and the density ratios are uniformly bounded. Hence Proposition~\ref{classicalproperties} applies with constants independent of \(j\). After enlarging \(C_0\) and decreasing \(r_0\), estimate \eqref{diameterup} holds uniformly for the approximating problems. The uniform \(C^{1,\alpha_0}\) estimate then implies
\begin{equation}\label{gradientconvergenceapproximation}
 Du_j\longrightarrow Du\qquad\text{locally uniformly in }\mathbb R^n.
\end{equation}
For fixed \(x_0\in\overline\Omega\), set \(y_j=Du_j(x_0)\) and
\[
 \Psi_j(x)=(x-x_0)\cdot(Du_j(x)-y_j).
\]
Then \(\Psi_j\to\Psi_{x_0}\) uniformly on \(\overline\Omega\).

Since \(t\mapsto p_\theta(t)\) is nondecreasing, each positive level set of
\[
\Psi_{x_0}(t\theta)=t\,p_\theta(t)
\]
meets every admissible ray in at most one point. Hence every positive level set of \(\Psi_{x_0}\) has Lebesgue measure zero by Fubini's theorem in polar coordinates. Since the boundary of a convex domain has Lebesgue measure zero, dominated convergence implies, for every \(h>0\),
\begin{equation}\label{convergenceapproximation}
 \int_{\mathbb R^n}\mathbf1_{\Omega_j}\mathbf1_{\{\Psi_j<h\}}f_j\,\dd x\longrightarrow\int_{\mathbb R^n}\mathbf1_\Omega\mathbf1_{\{\Psi_{x_0}<h\}}f\,\dd x.
\end{equation}
The smooth estimate \eqref{endpointinequality} for $(\Omega_j,\Omega_j^*,f_j,g_j,u_j)$ is uniform in \(j\). Passing to the limit in the integrated inequality proves \eqref{endpointinequality} for arbitrary bounded convex domains.

\end{proof}

\section{Blow-up analysis}
\label{secblowup}
We work directly with a blow-up sequence of the original transport problem.
Fix \((x_0,y_0)\in\overline\Omega\times\overline{\Omega^*}\) with
\(y_0=D u(x_0)\), and let \(h_j>0\) satisfy \(h_j\to0\). For each \(j\), let \(A_j\) be the positive definite linear map such that
\[
 B_1\subset A_j\bigl(S_{h_j}^c[u](x_0)-x_0\bigr)
 \subset B_{n^{3/2}}.
\]
Thus \(A_j(S_{h_j}^c[u](x_0)-x_0)\) is in John position. Set
\[
 \tilde x=A_j(x-x_0),\qquad
 \tilde y=h_j^{-1}A_j^{-t}(y-y_0),
\]
and
\[
 \Omega_j=A_j(\Omega-x_0),\qquad
 \Omega_j^*=h_j^{-1}A_j^{-t}(\Omega^*-y_0).
\]
The normalized potentials are
\begin{equation*}\label{afnormpotentials}
\begin{aligned}
 \tilde u_j(\tilde x)&=h_j^{-1}\bigl[\overline u(x_0+A_j^{-1}\tilde x)-\overline u(x_0)-y_0\cdot A_j^{-1}\tilde x\bigr],\\
 \tilde v_j(\tilde y)&=h_j^{-1}\bigl[\overline v(y_0+h_jA_j^t\tilde y)-\overline v(y_0)-x_0\cdot h_jA_j^t\tilde y\bigr].
\end{aligned}
\end{equation*}
At corresponding points,
\[
 \tilde y=D\tilde u_j(\tilde x),\qquad
 \tilde x=D\tilde v_j(\tilde y).
\]
After multiplying the two measures by the same constant \(a_j>0\), set
\[
 \rho_j(\tilde x)=a_j|\det A_j|^{-1}f(x_0+A_j^{-1}\tilde x),\qquad
 \rho_j^*(\tilde y)=a_jh_j^n|\det A_j|g(y_0+h_jA_j^t\tilde y),
\]
and
\[
 \dd\mu_j=\rho_j\mathbf1_{\Omega_j}\,\dd x,\qquad
 \dd\nu_j=\rho_j^*\mathbf1_{\Omega_j^*}\,\dd y.
\]
Then the minimal convex extensions \(\tilde u_j,\tilde v_j\) satisfy
\begin{equation}\label{datanormalized}
 (D\tilde u_j)_\#\mu_j=\nu_j,\qquad \tilde u_j=(\tilde v_j+I_{\overline{\Omega_j^*}})^*,\qquad \tilde v_j=(\tilde u_j+I_{\overline{\Omega_j}})^*.
\end{equation}
Here \(I_E\) denotes the convex indicator of \(E\). For a positive function \(\rho\) on a set \(E\) with \(\inf_E\rho>0\), define
\[
 \operatorname{osc}^{\rm rel}_E\rho=\frac{\sup_E\rho-\inf_E\rho}{\inf_E\rho}.
\]
By \eqref{eqdata},
\begin{equation}\label{densityratio}
 \frac{\sup_{\Omega_j}\rho_j}{\inf_{\Omega_j}\rho_j}+\frac{\sup_{\Omega_j^*}\rho_j^*}{\inf_{\Omega_j^*}\rho_j^*}\le D_0.
\end{equation}
Here one may take \(D_0=2\Lambda/\lambda\).
By \cite[Lemma~2.6 and Corollary~2.5]{JhaveriSavin2022}, the measures \(\mu_j\) and \(\nu_j\) have a common affine doubling constant depending only on \(n,D_0\). Moreover, by \eqref{globalupper}, we have $\tilde u_j,\tilde v_j\in C^1_{\rm loc}(\mathbb R^n).$

Under the above normalization, we assume throughout that
\begin{equation}\label{john}
\begin{gathered}
 0\in\overline{\Omega_j}\cap\overline{\Omega_j^*},\qquad \tilde u_j(0)=\tilde v_j(0)=0,\qquad D\tilde u_j(0)=D\tilde v_j(0)=0,\\
 K_j:=S_1^c[\tilde u_j](0)=\{x:\tilde u_j(x)-p_j\cdot x<1\},\qquad B_1\subset K_j\subset B_{n^{3/2}}.
\end{gathered}
\end{equation}
The barycenter of \(K_j\) is the origin. We also assume
\begin{equation}\label{massone}
 \mu_j(D_1(\tilde u_j,0))=\nu_j(D_1(\tilde v_j,0))=1,
\end{equation}
and, for every fixed \(R<\infty\),
\begin{equation}\label{relativeosc}
 \operatorname{osc}^{\rm rel}_{D_R(\tilde u_j,0)}\rho_j+\operatorname{osc}^{\rm rel}_{D_R(\tilde v_j,0)}\rho_j^*\longrightarrow0.
\end{equation}
Finally, assume that
\begin{equation}\label{Fflat}
 r^{-n}\mu_j(D_r(\tilde u_j,0))\longrightarrow1
\end{equation}
uniformly for \(r\) in compact subsets of \((0,\infty)\). The arguments below use only \eqref{datanormalized}--\eqref{Fflat}, and therefore apply to any normalized sequence satisfying these conditions.

The inclusions \eqref{include-D-S} and \eqref{include-S-Sc}, together with Lemmas~\ref{lemslope} and \ref{fixratio}, remain valid after normalization. We also use the following affine form of the subgradient image estimate. If \(K=S_h^c[\tilde u_j](0)\) and \(A_hK\) is in John position, then
\begin{equation}\label{gradnormalized}
 B_c\subset h^{-1}A_h^{-t}\bigl(\partial\tilde u_j(K)-p_h\bigr)\subset B_C,
\end{equation}
where \(c,C\) depend only on \(n,D_0\) \cite[Corollary 2.2]{Caffarelli1996Boundary}.

The first consequence is uniform compactness at every fixed extrinsic scale.

\begin{proposition}\label{noescape}
For every fixed \(0<R<\infty\) there is \(C_R<\infty\), independent of \(j\), such that whenever \(x\in D_R(\tilde u_j,0)\) and \(y=D\tilde u_j(x)\), one has
\begin{equation}\label{xybound}
 |x|+|y|\leq C_R.
\end{equation}
Consequently, suppose that 
\begin{equation}\label{noescapecondition}
 \sup_{x\in D_R(\tilde u_j,0)}\operatorname{dist}(x,L)\longrightarrow0
\end{equation}
for a linear subspace $L$. Then
\begin{equation}\label{normalproduct}
 \sup_{\substack{x\in D_R(\tilde u_j,0)\\y=D\tilde u_j(x)}}
 \left|P_{L^\perp}x\cdot P_{L^\perp}y\right|\longrightarrow0.
\end{equation}
\end{proposition}

\begin{proof}
By \eqref{include-S-Sc} and \eqref{include-D-S}
\[
 D_R(\tilde u_j,0)\subset S_{R^2}[\tilde u_j](0)
 \subset S^c_{CR^2}[\tilde u_j](0)\cap\Omega_j.
\]
Lemma~\ref{fixratio}, used in both directions, shows that the last centered sub-level set has inner and outer radii bounded in terms of $R$ and the fixed data. Lemma~\ref{lemslope} bounds $p_h$. Applying \eqref{gradnormalized} to that section bounds every corresponding subgradient $y$. This proves \eqref{xybound}. Condition \eqref{noescapecondition} gives \(P_{L^\perp}x\to0\) uniformly on \(D_R(\tilde u_j,0)\), and \eqref{normalproduct} follows from the bound for $y$.
\end{proof}

We next choose bounded truncations that contain a prescribed extrinsic ball and remain uniformly separated from the truncating level.
\begin{lemma}
\label{truncation11}
Under \eqref{john}, fix $R<\infty$. There are $H=H(R)>R^2$, $R_H<\infty$, and $d_H>0$, independent of $j$, such that the convex sets
\begin{equation*}
  \Omega_{j,H}:=\Omega_j\cap\{\tilde u_j<H\}
\end{equation*}
are uniformly bounded, and satisfy
\begin{equation*}
 D_R(\tilde u_j,0)\subset  \Omega_{j,H}\subset D_{R_H}(\tilde u_j,0),
 \qquad
 \operatorname{dist}\bigl(D_R(\tilde u_j,0),\Omega_j\cap\{\tilde u_j=H\}\bigr)
 \ge d_H.
\end{equation*}
\end{lemma}

\begin{proof}

After subtracting a supporting plane at the origin, \(\tilde u_j\ge0\). For \(y=D\tilde u_j(x)\), convexity gives
\begin{equation}\label{Ujxy}
 \tilde u_j(x)\le x\cdot y.
\end{equation}
Consequently
\[
 D_R(\tilde u_j,0)\subset\Omega_j\cap\{\tilde u_j<R^2\}.
\]
Fix \(H=R^2+1\). By \eqref{include-S-Sc},
\[
 \Omega_j\cap\{\tilde u_j<H+1\}
 \subset S^c_{C(H+1)}[\tilde u_j](0)\cap\Omega_j.
\]
Lemma~\ref{fixratio} and \eqref{john} give uniform inner and outer radii for the centered sub-level set on the right. Lemma~\ref{lemslope} and \eqref{gradnormalized} then imply
\begin{equation}\label{uniformbound}
 \sup_{\substack{x\in\Omega_j\\\tilde u_j(x)\le H}}
 \bigl(|x|+|D\tilde u_j(x)|\bigr)\le L_H
\end{equation}
with \(L_H\) independent of \(j\). In particular, \( \Omega_{j,H}\) is uniformly bounded.

If \(x\in D_R(\tilde u_j,0)\) and \(z\in\Omega_j\cap\{\tilde u_j=H\}\), then \([x,z]\subset\Omega_j\cap\{\tilde u_j\le H\}\). By \eqref{uniformbound}, \(\tilde u_j\) is \(L_H\)-Lipschitz on this segment, while \eqref{Ujxy} implies
\[
 L_H|z-x|\ge \tilde u_j(z)-\tilde u_j(x)>H-R^2=1.
\]
Thus \(d_H=L_H^{-1}\) works. Finally, for \(x\in  \Omega_{j,H}\), \eqref{uniformbound} implies
\[
 x\cdot D\tilde u_j(x)\le L_H^2.
\]
Taking \(R_H^2>L_H^2\) proves \( \Omega_{j,H}\subset D_{R_H}(\tilde u_j,0)\).
\end{proof}

The following elementary lemma identifies the support of a weak limit of measures  supported in uniformly bounded convex sets.
\begin{lemma}
\label{csptcollapse}
Let $E_j\subset B_R\subset\mathbb R^n$ be nonempty open convex sets and $c_j>0$. Suppose
\begin{equation}\label{measurelimit}
 \eta_j:=c_j\mathbf1_{E_j}\,\dd x\rightharpoonup\eta,
 \qquad
 0<\liminf_j\eta_j(E_j)\le
 \limsup_j\eta_j(E_j)<\infty.
\end{equation}
After passage to a subsequence, let $\overline E_j\to E$ in Hausdorff distance. Then
\begin{equation}\label{supportlimit}
 E=\operatorname{spt}\eta.
\end{equation}

The following localized formulation will be used. Suppose, in addition, that \(f_j\) is a continuous density on \(E_j\),
\[
 \varepsilon_j:=
 \sup_{E_j}\left|\frac{f_j}{c_j}-1\right|\longrightarrow0,
 \qquad
 \zeta_j:=f_j\mathbf1_{E_j}\,\dd x\rightharpoonup\zeta .
\]
If the mass bounds in \eqref{measurelimit} hold, then
\[
 \eta_j\rightharpoonup\zeta,
 \qquad
 \overline E_j\longrightarrow\operatorname{spt}\zeta
 \quad\text{in Hausdorff distance}.
\]
\end{lemma}

\begin{proof}
The inclusion $\operatorname{spt}\eta\subset E$ is immediate from Hausdorff convergence. Conversely, fix $x\in E$ and choose $x_j\in\overline E_j$ with $x_j\to x$. Replacing $x_j$ by an interior point at distance $o(1)$ does not change the argument. Put $D=2R$ and fix $0<\rho<D$. For $t=\rho/(4D)$, by convexity, we have
\[
 x_j+t(E_j-x_j)\subset E_j\cap B_{\rho/2}(x)
\]
for all large $j$. Hence
\begin{equation*}
 \eta_j(B_{\rho/2}(x))
 \ge t^n\eta_j(E_j).
\end{equation*}
Testing against a continuous cutoff which equals one on $B_{\rho/2}(x)$ and is supported in $B_\rho(x)$, and using \eqref{measurelimit}, gives $\eta(B_\rho(x))>0$. Since $\rho$ is arbitrary, $x$ belongs to $\operatorname{spt}\eta$, proving \eqref{supportlimit}.

Moreover,
\[
 \|\zeta_j-\eta_j\|_{\rm TV}
 \le\varepsilon_j\eta_j(E_j)\longrightarrow0,
\]
so \(\eta_j\rightharpoonup\zeta\). Every Hausdorff-convergent subsequence of \(\overline E_j\) has limit
\(\operatorname{spt}\zeta\) by the first part. By Blaschke compactness, we have
\[
 \overline E_j\longrightarrow\operatorname{spt}\zeta
\]
for the whole sequence.
\end{proof}

\section{Limits of normalized transports}
\label{sec5limit}

We pass to a limit in the normalized problems of Section~\ref{secblowup}, identify the limiting marginals, and reduce each lower-dimensional source limit to an optimal transport problem on its affine hull. The exact mass identity is preserved under this reduction. For an extended-valued convex function, \(\operatorname{cl}\) denotes its lower-semicontinuous closure. We first extract the limiting potentials and a locally finite transport plan.

\begin{lemma}
\label{diacompac}
Let a sequence of normalized transport problems satisfy \eqref{john}--\eqref{Fflat}. After passing to a subsequence there are a finite convex function  \(\tilde u_\infty:\mathbb R^n\to\mathbb R\), its Legendre transform
\[
 \tilde v_\infty:=\tilde u_\infty^*,
\]
and a locally finite transport plan \(\pi\) with the following properties.
\begin{enumerate}[label=\textup{(\roman*)}]
\item
\begin{equation}\label{Ujlimit}
 \tilde u_j\longrightarrow\tilde u_\infty
 \qquad\text{locally uniformly in }\mathbb R^n.
\end{equation}
Moreover, $\tilde u_\infty,\tilde v_\infty\ge 0, \tilde u_\infty(0)=\tilde v_\infty(0)=0$, and
\begin{equation}\label{dualidentification}
 \tilde v_\infty
 =\operatorname{cl}\bigl((\tilde v_\infty|_{\Omega_\infty^*})+I_{\Omega_\infty^*}\bigr),
 \qquad
 \tilde u_\infty=\tilde v_\infty^*.
\end{equation}
Let $\Omega_\infty^*:=\operatorname{int}(\operatorname{dom}\tilde v_\infty).$ Then, \(\Omega_\infty^*\ne\varnothing\), and for every \(Q\Subset \Omega_\infty^*\), we have
\begin{equation}\label{duallimit}
 Q\subset\Omega_j^*\quad\text{for all large }j,
 \qquad
 \tilde u_j^*\longrightarrow\tilde v_\infty
 \quad\text{uniformly on }Q.
\end{equation}

\item Consequently, for every \(h>0\),
\begin{equation}\label{ScUjlimit}
 \overline{S_h^c[\tilde u_j](0)}
 \longrightarrow
 \overline{S_h^c[\tilde u_\infty](0)}
\end{equation}
in Hausdorff distance.

\item Set $ \pi_j=(\operatorname{id},D\tilde u_j)_\#\mu_j$ and restrict each \(\pi_j\) to $G_R:=\{(x,y)\in\mathbb R^n\times\mathbb R^n:x\cdot y<R^2\}.$
Then, for every \(R>0\), 
\begin{equation}\label{limitplans}
 \pi_j\lfloor G_R\rightharpoonup \pi\lfloor G_R.
\end{equation}
Moreover,
\begin{equation}\label{piproperties}
 \spt\pi\subset
 \{(x,y):y\in\partial\tilde u_\infty(x)\}
 =\{(x,y):x\in\partial\tilde v_\infty(y)\},
 \qquad
 \pi(G_R)=R^n,
 \qquad
 \pi\{x\cdot y=R^2\}=0.
\end{equation}
\end{enumerate}
\end{lemma}

\begin{proof}
For \(h>0\), set \(K_{j,h}=S_h^c[\tilde u_j](0)\), and let \(p_{j,h}\) be the slope of its centering plane. Thus \(K_{j,1}=K_j\) and \(p_{j,1}=p_j\) in the notation of \eqref{john}. By \eqref{datanormalized}, the Alexandrov measure of \(\tilde u_j\) satisfies
\[
 \dd M_{\tilde u_j}
 =\frac{\rho_j}{\rho_j^*(D\tilde u_j)}\mathbf1_{\Omega_j}\,\dd x.
\]
Thus \(\spt M_{\tilde u_j}=\overline{\Omega_j}\), and \eqref{densityratio} gives a common affine doubling bound. Since \(0\in\overline{\Omega_j}\) by \eqref{john}, the geometric decay of centered sections  \cite[Lemma 2.2]{Caffarelli1996Boundary} yields \(s\in(0,1)\), independent of \(j\), such that
\[
 K_{j,sh}\subset\tfrac12K_{j,h}
 \qquad(h>0).
\]
Consequently, for every \(R>1\), there is \(H_R<\infty\), independent of \(j\), such that
\begin{equation}\label{geoexpansion}
 R K_j\subset K_{j,H_R}.
\end{equation}

Since \(B_1\subset K_j\) by \eqref{john}, \eqref{geoexpansion} gives \(B_R\subset K_{j,H_R}\). Convexity and \eqref{john} imply \(\tilde u_j\ge0\), while Lemma~\ref{lemslope} yields
\[
 |p_{j,H_R}|\le\frac{H_R}{R}.
\]
Hence
\begin{equation}\label{Ujbound}
 0\le\tilde u_j(x)
 \le H_R+p_{j,H_R}\cdot x
 \le2H_R
 \qquad(x\in B_R).
\end{equation}
Convexity and \eqref{Ujbound}, applied on a slightly larger ball, yield uniform local Lipschitz bounds. Arzel\`a--Ascoli and a diagonal argument prove \eqref{Ujlimit}. Since \(\tilde u_j\ge0\) and \(\tilde u_j(0)=0\), we have \(\tilde u_\infty\ge0\) and \(\tilde u_\infty(0)=0\). Therefore
\[
 \tilde v_\infty(0)=-\inf_{\mathbb R^n}\tilde u_\infty=0,
 \qquad
 \tilde v_\infty\ge0.
\]
The first identity in \eqref{dualidentification} follows from \cite[Theorem~7.5]{Rockafellar1970}. The second follows from the definition of \(\tilde v_\infty\) and the biconjugation theorem \cite[Theorem~12.2]{Rockafellar1970}.

To see that \(\Omega_\infty^*\ne\varnothing\), after passing to a further subsequence, assume that \(p_j\to p\). Taking \(h=1\) in \eqref{gradnormalized}, we obtain
\[
 p_j+B_c\subset\partial\tilde u_j(K_j).
\]
Fix \(y\in p+B_{c/2}\). For all large \(j\), there is \(x_j\in K_j\)
such that \(y\in\partial\tilde u_j(x_j)\). Since \(K_j\subset B_{n^{3/2}}\) and \(\tilde u_j\ge0\),
\[
 \tilde u_j^*(y)
 =x_j\cdot y-\tilde u_j(x_j)\le C.
\]
Hence, for every \(z\in\mathbb R^n\),
\[
 z\cdot y-\tilde u_j(z)\le C.
\]
Letting \(j\to\infty\) and then taking the supremum over \(z\), we obtain
\[
 \tilde v_\infty(y)=\tilde u_\infty^*(y)\le C.
\]
Thus
\[
 p+B_{c/2}\subset\operatorname{dom}\tilde v_\infty,
\]
and therefore \(\Omega_\infty^*\ne\varnothing\).

Fix \(Q\Subset \Omega_\infty^*\), and choose \(\delta>0\) such that
\[
 \overline Q+2\delta\overline B_1\Subset \Omega_\infty^*.
\]
Since \(\tilde v_\infty\) is continuous in \(\Omega_\infty^*\),
\[
 M:=\sup_{\overline Q+2\delta\overline B_1}\tilde v_\infty<\infty.
\]
For \(y\in\overline Q+\delta\overline B_1\), \(x\ne0\), and \(z=y+\delta x/|x|\), Fenchel's inequality implies
\begin{equation}\label{U-yx}
 \tilde u_\infty(x)-y\cdot x
 \ge\delta|x|-M.
\end{equation}
Choose \(R\) so large that \(\delta R-M\ge2\). By \eqref{Ujlimit} and \eqref{U-yx}, for all large \(j\),
\[
 \tilde u_j(Re)-y\cdot Re\ge1
\]
for every \(e\in\mathbb S^{n-1}\) and \(y\in\overline Q+\delta\overline B_1\). Convexity and \(\tilde u_j(0)=0\) then imply
\[
 \tilde u_j(te)-y\cdot te\ge\frac tR
 \qquad(t\ge R).
\]
Thus the suprema defining \(\tilde u_j^*\) and \(\tilde v_\infty\) on \(\overline Q+\delta\overline B_1\) are attained in one fixed ball. It follows that
\[
 \sup_{\overline Q+\delta\overline B_1}
 |\tilde u_j^*-\tilde v_\infty|
 \le
 \sup_{\overline B_R}|\tilde u_j-\tilde u_\infty|
 \longrightarrow0.
\]
In particular,
\[
 \overline Q+\delta\overline B_1
 \subset\operatorname{dom}\tilde u_j^*=\overline{\Omega_j^*}.
\]
Since \(\Omega_j^*\) is open and convex, \(Q\subset\Omega_j^*\), which proves \eqref{duallimit}.

For each fixed \(h>0\), \eqref{john} and Lemma~\ref{fixratio} give
\begin{equation}\label{Kjhshape}
 B_{c_h}\subset K_{j,h}\subset B_{C_h},
\end{equation}
with constants independent of \(j\). Hence Lemma~\ref{lemslope} bounds the centering slopes, and \eqref{ScUjlimit} follows directly from \eqref{Ujlimit}, preservation of the barycenter under Hausdorff convergence, and uniqueness of the centering plane.

We now prove \eqref{limitplans} and \eqref{piproperties}. For every \(R>0\), by \eqref{Fflat}, we have
\[
 \pi_j(G_R)
 =\mu_j\bigl(D_R(\tilde u_j,0)\bigr)
 \longrightarrow R^n.
\]
Moreover, by Proposition~\ref{noescape}, there is a compact set \(C_R\subset\mathbb R^n\times\mathbb R^n\), independent of \(j\), such that
\[
 \spt(\pi_j\lfloor G_R)\subset C_R.
\]
Thus the measures \(\pi_j\) have uniformly bounded mass on every compact subset of \(\mathbb R^n\times\mathbb R^n\). After passing to a diagonal subsequence, there is a locally finite measure \(\pi\) such that
\[
 \int\varphi\,\dd\pi_j\longrightarrow\int\varphi\,\dd\pi
 \qquad
 \text{for every }\varphi\in C_c(\mathbb R^n\times\mathbb R^n).
\]

We first prove the support assertion in \eqref{piproperties}. Suppose that
\[
 (x_j,y_j)\longrightarrow(x,y),
 \qquad
 y_j\in\partial\tilde u_j(x_j).
\]
For every \(z\in\mathbb R^n\),
\[
 \tilde u_j(z)
 \ge\tilde u_j(x_j)+y_j\cdot(z-x_j).
\]
Passing to the limit using \eqref{Ujlimit}, we obtain
\[
 \tilde u_\infty(z)
 \ge\tilde u_\infty(x)+y\cdot(z-x),
\]
and hence \(y\in\partial\tilde u_\infty(x)\). Therefore
\[
 \spt\pi\subset
 \{(x,y):y\in\partial\tilde u_\infty(x)\}.
\]
Since \(\tilde v_\infty=\tilde u_\infty^*\), by Legendre duality, we have
\[
 y\in\partial\tilde u_\infty(x)
 \quad\Longleftrightarrow\quad
 x\in\partial\tilde v_\infty(y).
\]
This proves the first assertion in \eqref{piproperties}.

It remains to identify the mass at every radius. At finite \(j\), each positive level of \((x,y)\mapsto x\cdot y\) has zero \(\pi_j\)-mass by the radial argument used in the proof of \eqref{convergenceapproximation}. Let \(0<R<k\) and \(0<\varepsilon<\min\{R,k-R\}\). Weak convergence of the restrictions at radius \(k\) yields that
\begin{align*}
 \pi\{x\cdot y<R^2\}
 &\ge(R-\varepsilon)^n,\\
 \pi\{x\cdot y\le R^2\}
 &\le(R+\varepsilon)^n.
\end{align*}
Letting \(\varepsilon\rightarrow 0\) proves
\begin{equation}\label{windowcost}
 \pi(G_R)=R^n,
 \qquad
 \pi\{x\cdot y=R^2\}=0.
\end{equation}
The restrictions obtained from different integer radii agree on their common continuity sets. They therefore define one locally finite plan \(\pi\), and \eqref{limitplans}--\eqref{piproperties} follow.
\end{proof}

The next lemma gives a direct geometric description of the density created by a lower-dimensional collapse. In the following, \(\operatorname{aff}(E)\) denotes the affine hull of \(E\), i.e., the smallest affine subspace containing \(E\).
\begin{lemma}\label{collapsedensity}
Let $E_j\subset B_R\subset\mathbb R^n$ be nonempty open convex sets and let $c_j>0$. Suppose that
\[
 \eta_j:=c_j\mathbf1_{E_j}\,\dd x\rightharpoonup\eta,
 \qquad
 0<\liminf_j\eta_j(E_j)\le
 \limsup_j\eta_j(E_j)<\infty,
\]
and that
\[
 \overline E_j\longrightarrow E=\spt\eta
 \qquad\text{in Hausdorff distance}.
\]
Set
\[
 L=\operatorname{aff}E,\qquad m=\dim L,
 \qquad \Omega'=\operatorname{relint}E.
\]
If $m=n$, then
\[
 \eta=c\mathbf1_{\Omega'}\,\dd x
\]
for some constant $c>0$. If $0<m<n$ and $q=n-m$, then
\begin{equation}\label{collapsedensityformula}
 \eta=\rho\mathbf1_{\Omega'}\,\dd\mathcal H^m\lfloor L,
 \qquad
 \rho^{1/q}\ \text{is positive and concave on }\Omega'.
\end{equation}
If $m=0$, then $\eta$ is a positive multiple of a Dirac mass.
\end{lemma}

\begin{proof}
The case $m=0$ is immediate. After a translation and a rotation, assume for the remaining cases that
\[
 L=\mathbb R^m\times\{0\},
\]
and denote the orthogonal projection onto $L$ by $P_L$.

Suppose first that $m=n$. Hausdorff convergence of bounded convex sets gives
\[
 |E_j\triangle E|\longrightarrow0,
 \qquad |E_j|\longrightarrow|E|>0.
\]
Since all the measures are supported in $B_R$, weak convergence also gives
\[
 c_j|E_j|=\eta_j(\mathbb R^n)\longrightarrow\eta(\mathbb R^n).
\]
Thus $c_j\to c>0$, and hence $ \eta=c\mathbf1_{\Omega'}\,\dd x,$ because the boundary of a full-dimensional convex set has Lebesgue measure zero.

Now let $0<m<n$ and put $q=n-m$. Set
\[
 F_{j,x}:=\{z\in L^\perp:x+z\in E_j\},
 \qquad x\in P_LE_j.
\]
By Fubini,
\begin{equation}\label{projectedsectiondensity}
 (P_L)_\#\eta_j=g_j\,\dd\mathcal H^m\lfloor L,
 \qquad
 g_j(x):=c_j\mathcal H^q(F_{j,x}).
\end{equation}
For $x_0,x_1\in P_LE_j$ and $0<t<1$, convexity gives
\[
 (1-t)F_{j,x_0}+tF_{j,x_1}
 \subset F_{j,(1-t)x_0+tx_1}.
\]
The $q$-dimensional Brunn--Minkowski inequality therefore shows that
\begin{equation}\label{sectionrootconcave}
g_j^{1/q}
 \quad\text{is nonnegative and concave on }P_LE_j.
\end{equation}

Since $E\subset L$ and $\overline E_j\to E$,
\[
 \delta_j:=\sup_{z\in E_j}|z-P_Lz|\longrightarrow0.
\]
Consequently, for every $\varphi\in C_c(\mathbb R^n)$,
\begin{align*}
 \left|\int\varphi\,\dd\eta_j
 -\int_L\varphi(x)g_j(x)\,\dd\mathcal H^m(x)\right|
 &\le \eta_j(E_j)
 \sup_{|z-z'|\le\delta_j}|\varphi(z)-\varphi(z')|\\
 &\longrightarrow0.
\end{align*}
Thus the projected measures in \eqref{projectedsectiondensity} converge weakly to $\eta$.

We next exclude concentration on an $\mathcal H^m$-null subset of $L$. Since $P_LE_j\to E$ in Hausdorff distance, there are $x_0\in\Omega'$ and $r>0$ such that $ B_r^L(x_0)\subset P_LE_j$ for all sufficiently large $j$. Fix $x\in P_LE_j$. For every $y\in B_r^L(x_0)$, concavity and nonnegativity give
\[
g_j^{1/q}\!\left(\frac{x+y}{2}\right)
 \ge\frac{g_j^{1/q}(x)+g_j^{1/q}(y)}2
 \ge\frac{g_j^{1/q}(x)}2,
\]
where $B_r^L(x_0)$ denotes a ball in $L$ with radius $r$ and center $x_0.$
By convexity
$$B_{\frac{r}{2}}^L(\frac{x+x_0}{2})=\{\frac{x+y}{2}: y\in B_r^L(x_0)\} \subset P_LE_j.$$
Hence
\[
 \eta_j(E_j)=\int_{P_L E_j}g_j\,\dd\mathcal H^m
 \ge \omega_m\left(\frac r2\right)^m
      \frac{g_j(x)}{2^q}.
\]
The uniform upper mass bound therefore yields
\begin{equation}\label{uniformsectionbound}
 \sup_{P_L E_j}g_j\le C.
\end{equation}
In particular, $ (P_L)_\#\eta_j\le C\mathcal H^m\lfloor L.$
Passing to the weak limit gives $ \eta\le C\mathcal H^m\lfloor L.$ Thus no mass can concentrate on a set of zero $\mathcal H^m$-measure.

By \eqref{uniformsectionbound}, and concavity of $g_j^{1/q}$, after passing to a subsequence,
\[
 g_j^{1/q}\longrightarrow w
 \qquad\text{locally uniformly on }\Omega',
\]
for some finite, nonnegative, concave function $w.$
 By \eqref{projectedsectiondensity} and the weak convergence just proved,
\[
 \eta\lfloor\Omega'=w^q\,\dd\mathcal H^m\lfloor\Omega'.
\]
Since \(\eta\le C\mathcal H^m\lfloor L\), \(\spt\eta=E\), and the relative boundary of a convex set has zero \(\mathcal H^m\)-measure, it follows that
\[
 \eta=w^q\mathbf1_{\Omega'}\,\dd\mathcal H^m\lfloor L.
\]
The function $w$ is not identically zero. By concavity and nonnegativity we have $w>0$ on $\Omega'$. Taking $\rho=w^q$ proves \eqref{collapsedensityformula}.
\end{proof}

The next lemma identifies the source marginal and proves local convergence of the source supports to its affine hull. In particular, the zero-dimensional alternative does not occur.

\begin{lemma}
\label{gmarconc}
Let \(\mu\) be the first marginal of the limiting transport plan $\pi$, which is obtained in Lemma~\ref{diacompac}. Then \(\mu\) is nonzero and locally finite, \(0\in\spt\mu\), and \(\spt\mu\) is convex. If
\[
 L=\operatorname{aff}(\spt\mu),
 \qquad
 m=\dim L,
\]
then \(L\) is a linear subspace and \(1\le m\le n\). Moreover, for every fixed \(R<\infty\),
\begin{equation}\label{noescapeDR}
 \sup_{x\in D_R(\tilde u_j,0)}\operatorname{dist}(x,L)
 \longrightarrow0.
\end{equation}
The source marginal has one of the following forms.
\begin{enumerate}[label=\textup{(\alph*)}]
\item If \(m=n\), then
\[
 \mu=\rho_\infty\mathbf1_{\Omega_\infty}\,\dd x
\]
for a constant \(\rho_\infty>0\) and the open convex set
\(\Omega_\infty=\operatorname{int}(\spt\mu)\).

\item If \(1\le m<n\), then, with \(q=n-m\),
\begin{equation}\label{lowerdimensionalsource}
 \mu=\rho(x)\mathbf1_{\Omega_\infty}\,\dd\mathcal H^m\lfloor L,
 \qquad
 \rho^{1/q}\ \text{is positive and concave on }\Omega_\infty,
\end{equation}
where \(\Omega_\infty=\operatorname{relint}(\spt\mu)\).
\end{enumerate}
\end{lemma}

\begin{proof}
We first prove local finiteness. By \eqref{john}, Lemma~\ref{fixratio}, and \eqref{include-S-Sc}, there is a dimensional constant \(r_*>0\) such that
\[
 \Omega_j\cap B_{r_*}\subset S_1[\tilde u_j](0).
\]
Together with \eqref{include-D-S}, this yields
\[
 \tfrac12(\Omega_j\cap B_{r_*})\subset D_1(\tilde u_j,0).
\]
Convexity, \(0\in\overline{\Omega_j}\), and \eqref{densityratio} imply
\[
 \mu_j(B_{r_*})
 \le2^nD_0\,
 \mu_j\bigl(\tfrac12(\Omega_j\cap B_{r_*})\bigr)
 \le2^nD_0.
\]
The common doubling bound then shows that
\[
 \sup_j\mu_j(B_M)<\infty
 \qquad(M<\infty).
\]
Combining this with \eqref{limitplans}, we conclude that \(\mu\) is locally finite.

Choose \(1\le R_1<R_2<\cdots\to\infty\). By Lemma~\ref{truncation11}, the heights may be chosen increasingly so that
\begin{equation}\label{nestedwindow}
 D_{R_k}(\tilde u_j,0)
 \subset\Omega_{j,H_k}
 \subset D_{\widehat R_k}(\tilde u_j,0),
\end{equation}
where \(\widehat R_k<\infty\) is independent of \(j\), and the sets \(\Omega_{j,H_k}\) are nested in \(k\). Set
\begin{equation}\label{fjapprox}
 c_{k,j}:=\inf_{\Omega_{j,H_k}}\rho_j>0,
 \qquad
 \varepsilon_{k,j}:=
 \sup_{\Omega_{j,H_k}}
 \left(\frac{\rho_j}{c_{k,j}}-1\right)
 =\operatorname{osc}^{\rm rel}_{\Omega_{j,H_k}}\rho_j
 \longrightarrow0
 \qquad(j\to\infty).
\end{equation}
Here the convergence follows from
\(\Omega_{j,H_k}\subset D_{\widehat R_k}(\tilde u_j,0)\) and \eqref{relativeosc}. Define
\begin{equation}\label{etajR}
 \eta_{k,j}=c_{k,j}\mathbf1_{\Omega_{j,H_k}}\,\dd x,
 \qquad
 \zeta_{k,j}=\mu_j\lfloor\Omega_{j,H_k}.
\end{equation}

The two inclusions in \eqref{nestedwindow}, together with \eqref{massone} and \eqref{Fflat}, give uniform positive lower and finite upper bounds for \(\zeta_{k,j}(\mathbb R^n)\). Since
\[
 \eta_{k,j}\le\zeta_{k,j}
 \le(1+\varepsilon_{k,j})\eta_{k,j},
\]
the same mass bounds hold for \(\eta_{k,j}\), and
\[
 \|\eta_{k,j}-\zeta_{k,j}\|_{\mathrm{TV}}
 \le\varepsilon_{k,j}\eta_{k,j}(\mathbb R^n)
 \longrightarrow0.
\]
After a diagonal extraction,
\[
 \eta_{k,j}\rightharpoonup\zeta_k,
 \qquad
 \zeta_{k,j}\rightharpoonup\zeta_k
\]
for every \(k\). Since the truncations are nested, we have
\begin{equation}\label{nestedmeasures}
 \zeta_k\le\zeta_\ell\qquad(k\le\ell).
\end{equation}
By Lemma~\ref{csptcollapse},
\begin{equation}\label{truncationsupportlimit}
 \overline{\Omega_{j,H_k}}
 \longrightarrow S_k:=\spt\zeta_k
 \qquad\text{in Hausdorff distance}
\end{equation}
for every \(k\). Thus each \(S_k\) is convex. Moreover, \(0\in S_k\), because \(0\in\overline{\Omega_{j,H_k}}\), and
\eqref{nestedmeasures} gives \(S_k\subset S_\ell\) for \(k\le\ell\).

Let \(\mu_R\) be the first marginal of \(\pi\lfloor G_R\). Taking weak limits in the measure inclusions induced by \eqref{nestedwindow}, and using Lemma~\ref{diacompac} for the two outer terms, we obtain
\begin{equation}\label{marginalsandwich}
 \mu_{R_k}\le\zeta_k\le\mu_{\widehat R_k}.
\end{equation}
Since \(\mu_R\uparrow\mu\) as \(R\to\infty\), it follows that
\begin{equation}\label{musupzeta}
 \mu=\sup_k\zeta_k.
\end{equation}
Consequently,
\[
 \spt\mu=\overline{\bigcup_k S_k}.
\]
Hence \(\spt\mu\) is convex, \(0\in\spt\mu\), and
\(L=\operatorname{aff}(\spt\mu)\) is a linear subspace.

For each \(k\), \eqref{nestedwindow} and \eqref{truncationsupportlimit} give
\[
 \sup_{x\in D_{R_k}(\tilde u_j,0)}\operatorname{dist}(x,L)
 \le d_H(\overline{\Omega_{j,H_k}},S_k)
 \longrightarrow0.
\]
Since \(R_k\to\infty\) and the sets \(D_R(\tilde u_j,0)\) are nested in \(R\), this proves \eqref{noescapeDR} for every fixed \(R<\infty\).

It remains to rule out \(m=0\). Suppose that \(L=\{0\}\). Since
\[
 \spt(\pi_j\lfloor G_1)
 \subset \spt\pi_j\cap\{(x,y):x\cdot y\le1\},
\]
\eqref{noescapeDR} with \(R=2\) gives
\[
 \sup_{(x,y)\in\spt(\pi_j\lfloor G_1)}|x|
 \longrightarrow0.
\]
On the other hand, Proposition~\ref{noescape} gives
\[
 \sup_j\sup_{(x,y)\in\spt(\pi_j\lfloor G_1)}|y|<\infty.
\]
Consequently,
\[
 \sup_{(x,y)\in\spt(\pi_j\lfloor G_1)}
 |x\cdot y|
 \longrightarrow0.
\]
Thus, for every fixed \(0<r<1\),
\[
 \spt(\pi_j\lfloor G_1)\subset G_r
\]
for all sufficiently large \(j\). Since \(\pi_j(G_1)=1\), this implies
\[
 \mu_j(D_r(\tilde u_j,0))
 =\pi_j(G_r)=1,
\]
contradicting \eqref{Fflat}, because
\[
 \mu_j(D_r(\tilde u_j,0))\longrightarrow r^n<1.
\]
Therefore \(m\ge1\).

Set \(S=\spt\mu\). Since \(S=\overline{\bigcup_kS_k}\) and the sets \(S_k\) are increasing, there is \(k_0\) such that $\operatorname{aff}S_k=L$
for $k\ge k_0$. Moreover, by convexity we have that every compact subset of \(\operatorname{relint}S\) is contained in \(\operatorname{relint}S_k\) for all sufficiently large \(k\). For \(k\ge k_0\), apply Lemma~\ref{collapsedensity} to
\(E_j=\Omega_{j,H_k}\) and \(\eta_j=\eta_{k,j}\). Denote $\Omega_k=\operatorname{relint}S_k.$

If \(m=n\), set \(\Omega_\infty=\operatorname{int}S\). The lemma gives
\[
 \zeta_k=a_k\mathbf1_{\Omega_k}\,\dd x,
 \qquad a_k>0.
\]
By \eqref{nestedmeasures} we have that \(a_k\) is nondecreasing. Choosing a ball \(B\Subset\Omega_{k_0}\), we have
\[
 a_k|B|=\zeta_k(B)\le\mu(B),
 \qquad k\ge k_0,
\]
so \(a_k\uparrow\rho_\infty<\infty\). Therefore \eqref{musupzeta} gives
\[
 \mu=\rho_\infty\mathbf1_{\Omega_\infty}\,\dd x.
\]

Now assume \(1\le m<n\) and put \(q=n-m\). Lemma~\ref{collapsedensity} gives
\[
 \zeta_k=\rho_k\mathbf1_{\Omega_k}\,\dd\mathcal H^m\lfloor L,
\]
with \(\rho_k^{1/q}\) positive and concave on \(\Omega_k\). Since \(S_k\subset S_\ell\) and
\(\zeta_k\le\zeta_\ell\), the continuous representatives satisfy
\[
\rho_k^{1/q}\le \rho_\ell^{1/q}\qquad\text{on }\Omega_k,
 \qquad k\le\ell.
\]
Set \(\Omega_\infty=\operatorname{relint}S\). Then
\[
 w(x):=\lim_{k\to\infty}\rho_k^{1/q}(x),
 \qquad x\in\Omega_\infty,
\]
where the limit is taken for sufficiently large \(k\) such that \(x\in\Omega_k\). It is well defined and concave, initially with values in \((0,\infty]\). It is finite everywhere. Indeed, if \(w(x_0)=\infty\), choose a ball \(B\Subset\Omega_\infty\) and \(0<t<1\). For all sufficiently large \(k\), concavity gives
\[
 \rho_k^{1/q}((1-t)x_0+ty)\ge(1-t)\rho_k^{1/q}(x_0)
 \qquad(y\in B).
\]
Letting \(k\to\infty\) would make \(w\) infinite on the open set \((1-t)x_0+tB\). Since \(\mu=\sup_k\zeta_k\), monotone convergence would then give infinite \(\mu\)-mass on a compact subset of this set, contradicting local finiteness. Thus \(w\) is positive and finite on \(\Omega_\infty\).

Extend \(\rho_k\) by zero outside \(\Omega_k\). These densities increase almost everywhere to \(w^q\mathbf1_{\Omega_\infty}\). Hence monotone convergence in \eqref{musupzeta} yields that
\[
 \mu=w^q\mathbf1_{\Omega_\infty}\,\dd\mathcal H^m\lfloor L.
\]
Taking \(\rho=w^q\) proves \eqref{lowerdimensionalsource}.
\end{proof}

We next identify the target marginal.

\begin{lemma}
\label{fibredu}
Assume \eqref{john}--\eqref{Fflat}, and suppose that, for an \(m\)-dimensional subspace \(L\), with \(1\le m<n\), \eqref{noescapecondition} holds on \(D_R(\tilde u_j,0)\) for every fixed \(R\). Let \(\mu\) and \(\nu\) be the first and second marginals of \(\pi\), respectively, and write, as in \eqref{lowerdimensionalsource},
\[
 \mu
 =\rho\,\mathbf1_{\Omega_\infty}\,
   \mathcal H^m\lfloor L,
 \qquad
 \Omega_\infty=\operatorname{relint}(\spt\mu).
\]
After passing to a subsequence, the following statements hold.

\begin{enumerate}
\item Let \(\Omega_\infty^*\) be as in Lemma~\ref{diacompac}. Then there is a constant \(\rho_\infty^*>0\) such that
\begin{equation}\label{targetmarginal}
 \nu=\rho_\infty^*\,\mathcal L^n\lfloor\Omega_\infty^*.
\end{equation}
Moreover, the normalized target measures converge in total variation on every compact subset of \(\Omega_\infty^*\) to the measure in
\eqref{targetmarginal}.

\item Set $\Omega_L^*=P_L \Omega_\infty^*.$ After identifying \(L\) isometrically with \(\mathbb R^m\), there are dual convex potentials \(\phi\) on \(\Omega_\infty\) and \(\varphi\) on \(\Omega_L^*\) such that
\begin{equation}\label{vfiber}
 \tilde u_{\infty}(x)=\phi(x),
 \qquad\tilde v_\infty(y)=\varphi(P_Ly),
 \qquad x\in \Omega_\infty,\  y\in \Omega_\infty^*.
\end{equation}
\item  The projected target density $$ \rho_L^*(\xi)
 =\rho_\infty^*\mathcal H^q
 \bigl(\Omega_\infty^*\cap(\xi+L^\perp)\bigr)$$ is finite and positive for $\xi\in \Omega_L^*,$ and \((\rho_L^*)^{1/q}\) is concave in \(\Omega_L^*\). The minimal extensions satisfy
$$
(D\phi)_\#\bigl(\rho\,\mathcal H^m\lfloor \Omega_\infty\bigr)
 =\rho_L^*\mathbf1_{\Omega_L^*}\,\dd\xi,
 \qquad
 (D\varphi)_\#
 \bigl(\rho_L^*\mathbf1_{\Omega_L^*}\,\dd\xi\bigr)
 =\rho\,\mathcal H^m\lfloor \Omega_\infty.
$$
Moreover, for every \(r>0\),
\begin{equation}\label{reducedpush}
 \int_{\{x\in \Omega_\infty:x\cdot D\phi(x)<r^2\}}
 \rho(x)\,\dd\mathcal H^m(x)
 =
 \int_{\{\xi\in \Omega_L^*:\xi\cdot D\varphi(\xi)<r^2\}}
 \rho_L^*(\xi)\,\dd\xi
 =r^n.
\end{equation}
\end{enumerate}
\end{lemma}

\begin{proof}
Set
\begin{equation}\label{bj}
 b_j:=\inf_{D_1(\tilde v_j,0)}\rho_j^*.
\end{equation}
We first prove uniform bounds for \(b_j\). With \(K_j\) and \(p_j\) as in \eqref{john}, \eqref{gradpolar} and \eqref{john} give
\begin{equation}\label{dualKj}
 p_j+B_{c_0}\subset\partial  \tilde u_j(K_j)\subset\overline{\Omega_j^*},
 \qquad |p_j|\le1,
\end{equation}
for some constant $c_0$ depending only on $n.$ Hence, 
\begin{equation}\label{shrinkOmega}
 p_j+B_{c_0/2}\subset\Omega_j^*.
\end{equation}
For every $y$ in this ball, the corresponding point \(x=D\tilde v_j(y)\) belongs to $K_j$. The bounds on $p_j$, $K_j$, and the ball give $|x|+|y|\le C$, and therefore $x\cdot y<R_0^2$ for a fixed $R_0\ge1$. Thus
\begin{equation}\label{shrinkDR}
 p_j+B_{c_0/2}\subset D_{R_0}(\tilde v_j,0).
\end{equation}
 Consequently, by \eqref{relativeosc} we have
\begin{equation}\label{bjapprox}
 \sup_{D_{R_0}(\tilde v_j,0)}\rho_j^*
   =(1+o(1))b_j.
\end{equation}
The ball in \eqref{shrinkDR} and $\nu_j(D_{R_0}(\tilde v_j, 0))=\mu_j(D_{R_0}(\tilde u_j, 0))\to R_0^n$ (follows by \eqref{Fflat}) give $b_j\le C$. Conversely, Proposition~\ref{noescape} implies that $D_1(\tilde v_j,0)$ is contained in a fixed Euclidean ball independent of $j$, and
\[
 1=\nu_j(D_1(\tilde v_j,0))
 \le (1+o(1))b_j|D_1(\tilde v_j,0)|\le Cb_j,
\]
where the first equality follows from the assumption \eqref{massone}. This proves $0<c\le b_j\le C$. Pass to a subsequence with $b_j\to \rho^*_{\infty}>0,$ for some constant $ \rho^*_{\infty}.$

We next identify the density on compact subsets. Fix \(Q\Subset  \Omega_\infty^*\), and choose \(Q'\) with \(Q\Subset Q'\Subset  \Omega_\infty^*\). By \eqref{duallimit}, for all large \(j\),
\begin{equation}\label{duallimit2}
 Q'\subset\Omega_j^*,
 \qquad \tilde v_j\longrightarrow \tilde v_\infty \quad\hbox{uniformly on }Q'.
\end{equation}
By convexity we have
\begin{equation}\label{subgradbound}
 \sup_{y\in Q}|\partial\tilde v_j(y)|\le C_Q.
\end{equation}
Thus $|y\cdot D\tilde v_j(y)|\le C_Q$ for every $y\in Q$, so
\(Q\subset D_R(\tilde v_j,0)\) for some fixed \(R\). By
\eqref{relativeosc}, we have
\begin{equation}\label{densityapprox}
 \sup_{Q}\left|\frac{\rho^*_j}{b_j}-1\right|
 \longrightarrow0.
\end{equation}
Hence the target marginals converge on every compact subset of \( \Omega_\infty^*\), indeed in local total variation, to $\rho^*_{\infty}\mathbf1_{ \Omega_\infty^*}\,\dd y$.

It remains to identify the second marginal on the boundary of \(\Omega_\infty^*\). For \(R>0\), set
\[
 \nu_{j,R}:=(\operatorname{pr}_2)_\#(\pi_j\lfloor G_R),
 \qquad
 \nu_R:=(\operatorname{pr}_2)_\#(\pi\lfloor G_R).
\]
By \eqref{limitplans},
\[
 \nu_{j,R}\rightharpoonup\nu_R.
\]
Moreover,
\[
 \nu_{j,R}
 =\rho_j^*\mathbf1_{D_R(\tilde v_j,0)}\,\dd y
 \le C_R\,\dd y,
\]
where \(C_R\) is independent of \(j\). Passing to the limit, we have
\[
 \nu_R\le C_R\,\dd y.
\]
Since \(G_R\uparrow\mathbb R^n\times\mathbb R^n\), we have \(\nu_R\uparrow\nu\), where \(\nu\) is the second marginal of \(\pi\). Consequently,
\[
 \nu\ll\mathcal L^n.
\]

On the other hand, the support relation in \eqref{piproperties} gives
\[
 \nu\bigl(\mathbb R^n\setminus
 \operatorname{dom}\tilde v_\infty\bigr)=0.
\]
Since
\[
 \operatorname{dom}\tilde v_\infty\setminus\Omega_\infty^*
 \subset\partial\Omega_\infty^*
\]
and the boundary of a full-dimensional convex set has Lebesgue measure zero, it follows that
\[
 \nu(\mathbb R^n\setminus\Omega_\infty^*)=0.
\]

Finally, let \(\varphi\in C_c(\Omega_\infty^*)\). The local boundedness of the subgradients of \(\tilde v_j\) and \(\tilde v_\infty\) allows us to choose \(R\) such that all the transport pairs whose second coordinate lies in \(\operatorname{spt}\varphi\) belong to \(G_R\). Therefore, by \eqref{densityapprox},
\[
 \int\varphi\,\dd\nu
 =\lim_{j\to\infty}\int\varphi(y)\rho_j^*(y)\,\dd y
 =\rho_\infty^*\int\varphi(y)\,\dd y.
\]
Thus
\[
 \nu=\rho_\infty^*
 \mathbf1_{\Omega_\infty^*}\,\dd y.
\]

We now factorize the dual potential along the directions in \(L^\perp\). The first marginal of the limit plan is supported on $L$. By \eqref{piproperties}, we have 
\begin{equation}\label{DVinL}
 D\tilde v_\infty(y)\in L\qquad\text{for a.e. }y\in  \Omega_\infty^*.
\end{equation}
By convexity, for any \(y\in  \Omega_\infty^*, p\in \partial \tilde v_\infty(y)\) and \(e\in L^{\perp}\), one has $p\cdot e=0.$ Hence \(\tilde v_\infty\) is constant on every convex fiber \(\Omega^*_{\infty,\xi}:= \Omega_\infty^*\cap(\xi+L^\perp)\). There is a finite convex function $\varphi$ on the open convex projected domain $\Omega_L^*=P_L \Omega_\infty^*$ such that
\begin{equation}
 \tilde v_\infty(y)=\varphi(P_Ly),\qquad y\in  \Omega_\infty^*.
\end{equation}

Define
\[
 \phi(x):=
 \sup_{\xi\in\Omega_L^*}
 \{x\cdot\xi-\varphi(\xi)\},
 \qquad x\in L.
\]
For \(x\in L\), \eqref{dualidentification} and the preceding factorization give
\begin{align*}
 \tilde u_\infty(x)
 &=\sup_{y\in\Omega_\infty^*}
   \{x\cdot y-\tilde v_\infty(y)\}\\
 &=\sup_{y\in\Omega_\infty^*}
   \{x\cdot P_Ly-\varphi(P_Ly)\}\\
 &=\sup_{\xi\in\Omega_L^*}
   \{x\cdot\xi-\varphi(\xi)\}
 =\phi(x).
\end{align*}
Thus both identities in \eqref{vfiber} hold.

We next integrate along the \(L^\perp\)-fibers. Put \(q=n-m\) and identify \(L\) with \(\mathbb R^m\). Fubini gives
\[
 (P_L)_\#
 \bigl(\rho_\infty^*\mathbf1_{\Omega_\infty^*}\,\dd y\bigr)
 =\rho_L^*(\xi)\,\dd\xi,
 \qquad
 \rho_L^*(\xi)
 =\rho_\infty^*\mathcal H^q(\Omega^*_{\infty,\xi})
 =\rho_\infty^*\mathcal H^q
 \bigl(\Omega_\infty^*\cap(\xi+L^\perp)\bigr).
\]
For $\xi_0,\xi_1\in \Omega_L^*$ and $0<t<1$, convexity gives
\begin{equation}\label{fiberincl}
 (1-t) \Omega^*_{\infty,\xi_0}+t \Omega^*_{\infty,\xi_1}
 \subset  \Omega^*_{\infty,(1-t)\xi_0+t\xi_1}.
\end{equation}
The $q$-dimensional Brunn--Minkowski inequality therefore shows that $(\rho_L^*)^{1/q}$ is concave wherever it is finite.

It remains to exclude infinite fibers without assuming that \( \Omega_\infty^*\) is bounded. If \(Q\Subset\Omega_L^*\), by local boundedness of \(\partial\varphi\), we have
\[
 \sup\{|\xi\cdot x|:\xi\in Q,\ x\in\partial\varphi(\xi)\}<\infty.
\]
By \eqref{vfiber}, almost every point of $ \Omega_\infty^*\cap P_L^{-1}(Q)$ lies in \(\{y\in  \Omega_\infty^*:y\cdot D\tilde v_{\infty}(y)<R^2\}\) for some fixed \(R\). That set has finite limiting mass, and hence
\begin{equation}\label{kfinite}
 \int_Q \rho_L^*\,d\xi<\infty.
\end{equation}
Thus $\rho_L^*<\infty$ almost everywhere. Suppose that \(\rho_L^*(\xi_0)=\infty\) at an interior point. Fix a ball \(B\Subset\Omega_L^*\) not containing \(\xi_0\). For every fixed \(t\in(0,1)\), by \eqref{fiberincl}, the fiber over \((1-t)\xi_0+t\eta\) has infinite volume for every \(\eta\in B\). These base points form a set of positive measure, which contradicts \eqref{kfinite}. Therefore every fiber over \(\xi\in\Omega_L^*\) is bounded and
\begin{equation*}
 0<\rho_L^*(\xi)<\infty\qquad(\xi\in\Omega_L^*).
\end{equation*}
In particular $(\rho_L^*)^{1/q}$ is a finite concave function on $\Omega_L^*$.

Finally, we identify the projected transport plan and preserve the mass identity. Let $\widehat\pi$ be the image of the limiting transport plan under $(x,y)\mapsto(x,P_Ly)$. Its marginals are $\rho\mathbf1_{\Omega_{\infty}}\,d\mathcal H^m\lfloor L$ and $\rho_L^*\,d\xi$. 
By \eqref{piproperties} and \eqref{vfiber}, it is concentrated on $x\in\partial\varphi(\xi)$. Since $\rho_L^*\,d\xi$ is absolutely continuous and $\varphi$ is differentiable almost everywhere,
\begin{equation}\label{projectedtransport}
 \widehat\pi=(D\varphi,\operatorname{id})_\#(\rho_L^*\,\dd\xi),
 \qquad (D\varphi)_\#(\rho^*_L\,\dd\xi)=\rho\mathbf1_{\Omega_{\infty}}\,\dd\mathcal H^m\lfloor L.
\end{equation}

Since the first marginal of \(\pi\) is supported on \(L\),
\[
 x\cdot P_Ly=x\cdot y\qquad\text{for \(\pi\)-a.e. }(x,y).
\]
Consequently, by Lemma~\ref{diacompac}(iii), we have for every \(r>0\),
\[
 \widehat\pi\{(x,\xi):x\cdot\xi<r^2\}=r^n.
\]
Using \eqref{projectedtransport}, we obtain \eqref{reducedpush}. Thus the mass identity is preserved.
\end{proof}

\section{Rigidity of normalized limits}
\label{sec6rigidity}
We prove the equality case for the weighted transports obtained in Section~\ref{sec5limit}, and then lift the resulting homogeneity to the original normalized limit. The concavity of \(\rho^{1/q}\) first yields a radial inequality with a sharp equality condition.
\begin{lemma}
Let \(q>0\). Let \(C\) be convex with \(0\in\overline C\), and suppose that \(\rho^{1/q}\) is finite, nonnegative, and concave on \(C\). Then
\begin{equation}\label{radialq}
 x\cdot D\rho(x)\le q\rho(x)
\end{equation}
at every differentiability point. On every admissible ray, the function \(r\mapsto r^{-q}\rho(r\theta)\) is nonincreasing. It is constant on an interval precisely when \(\rho(tx)=t^q\rho(x)\) there.
\end{lemma}

\begin{proof}
Set \(w=\rho^{1/q}\). Fix an admissible ray and \(0<\delta<s<r\). Concavity and nonnegativity imply
\[
 w(s\theta)\ge\frac{r-s}{r-\delta}w(\delta\theta)+\frac{s-\delta}{r-\delta}w(r\theta)\ge\frac{s-\delta}{r-\delta}w(r\theta).
\]
Letting \(\delta\rightarrow 0\), we have \(w(s\theta)/s\ge w(r\theta)/r\). Hence \(r^{-q}\rho(r\theta)\) is nonincreasing, and differentiation proves \eqref{radialq}. Now, \(r^{-q}\rho(r\theta)\) is locally absolutely continuous, hence, equality in \eqref{radialq} almost everywhere implies  that \(r^{-q}\rho(r\theta)\) is a constant, which implies \(\rho(tx)=t^q\rho(x)\).
\end{proof}

The limiting transports below may have infinite total mass. We shall use Proposition~\ref{prop:localfinite}, proved in Appendix~\ref{appendix}, which implies strict convexity, local \(C^1\) regularity, and the homeomorphism property for such locally finite transports. Hölder continuity under the assumptions above follows from the interior regularity theory for the Monge--Ampère equation.

Let $m\ge1$. Let $C,C^*\subset\mathbb R^m$ be open convex domains and let \(\phi\) on \(C\) and \(\varphi\) on \(C^*\) be the corresponding convex potential functions whose extensions satisfy \eqref{uvextensions} and \eqref{uvdual} in Section \ref{section2}. Since $\phi$ and $\varphi$ are Legendre duals of each other within the convex domains \(C\) and \(C^*\), we can express this minimal convex extension in the following equivalent form.
\begin{equation}\label{qrigid-extensions}
\begin{aligned}
      \underline{\phi}(x)&=\sup\{\ell(x):\ell\text{ is affine},\ \ell\le \phi\text{ in }C,
                         \ D\ell\in\overline{C^*}\},\\
     \underline{\varphi}(y)&=\sup\{\ell(y):\ell\text{ is affine},\ \ell\le \varphi\text{ in }C^*,
                         \ D\ell\in\overline C\}.
\end{aligned}
\end{equation}
For simplicity of notation, we continue to denote the extended functions by $\phi$ and $\varphi$.

\begin{lemma}
\label{qrigid-local-calculus}
Let $m\ge1$ and $q\ge0$. Let $C,C^*,\phi$ and $\varphi$ be given as above. Assume that \(D\phi\) transports \(\rho\,\dd x\) to \(\rho^*\,\dd y\). If \(q>0\), assume that \(\rho^{1/q}\) and \((\rho^*)^{1/q}\) are positive and concave in \(C\) and \(C^*\), respectively. If \(q=0\), assume that \(\rho,\rho^*\) are positive constants. Suppose also that \((x_c,y_c)\in\overline C\times\overline{C^*}\) satisfies
\[
 y_c\in\partial \phi(x_c),\qquad
 x_c\in\partial \varphi (y_c),\qquad
  \phi(x_c)+ \varphi (y_c)=x_c\cdot y_c,
\]
and, for every \(h>0\), the two sets
\[
 \bigl\{x\in C:(x-x_c)\cdot(D\phi(x)-y_c)<h\bigr\},
 \qquad
 \bigl\{y\in C^*:(y-y_c)\cdot(D\varphi(y)-x_c)<h\bigr\}
\]
are bounded and have finite mass. Then
\begin{equation*}
 \phi,\varphi\ \hbox{are strictly convex},\qquad
 \phi\in C^{1,\alpha}_{\rm loc}(C),\quad
 \varphi\in C^{1,\alpha}_{\rm loc}(C^*)
\end{equation*}
with a possibly compact-set-dependent exponent $\alpha>0$, and the following facts hold.

\begin{enumerate}[label=\textup{(\roman*)}]
\item The two measures are doubling. More precisely, if $E$ is a bounded ellipsoid centered at a point of \(\overline C\), or respectively of \(\overline{C^*}\), then
\begin{equation}\label{qrigid-doubling}
 \int_E \rho\,\dd x\le 2^{m+q}\int_{\frac12E}\rho\,\dd x,
 \qquad
 \int_E \rho^*\,\dd y\le 2^{m+q}\int_{\frac12E}\rho^*\,\dd y.
\end{equation}
When $q=0$, the exponent on the right is $m$. In these inequalities \(\rho\) and \(\rho^*\) are extended by zero outside \(C\) and \(C^*\), respectively.

\item The maps
\begin{equation}\label{qrigid-homeomorphisms}
 T:=D\phi:C\longrightarrow C^*,\qquad
 S:=D\varphi:C^*\longrightarrow C
\end{equation}
are mutually inverse homeomorphisms. In particular,
\begin{equation*}
 K\Subset C\Longrightarrow T(K)\Subset C^*,\qquad
 K^*\Subset C^*\Longrightarrow S(K^*)\Subset C.
\end{equation*}

\item After possibly decreasing the local exponent,
\begin{equation*}
 \phi\in C^{2,\alpha'}_{\rm loc}(C),\qquad
 \varphi\in C^{2,\alpha'}_{\rm loc}(C^*).
\end{equation*}
At corresponding points $y=T(x)$,
\begin{equation}\label{qrigid-legendre}
 H:=D^2\phi(x)>0,
 \qquad D^2\varphi(y)=H^{-1}.
\end{equation}

\item After translating \(x_c,y_c\) to the origin, denote
\[
 \Psi(x)=x\cdot T(x),\qquad
 \Psi^*(y)=y\cdot S(y),\qquad
 M(h)=\int_{\{\Psi<h\}}\rho.
\]
Then $\Psi_\#(\rho\,\dd x)$ is absolutely continuous on compact subintervals of $(0,\infty)$, $M$ is locally absolutely continuous, and
\begin{equation*}
 M'(h)=\int_{\{\Psi=h\}}\frac{\rho}{|\nabla\Psi|}
          \,d\mathcal H^{m-1}
 \quad\text{for a.e. }h>0.
\end{equation*}
\end{enumerate}
\end{lemma}

\begin{proof}

Assume first that $q>0$ and write $w=(\rho)^{1/q}$.  If $E$ is centered at $z\in\overline C$, the map \(x\mapsto z+(x-z)/2\) sends \(E\cap C\) into \(\frac12E\cap C\). If $z\in C$, by concavity and nonnegativity, we have
\[
 w\left(\frac{x+z}{2}\right)\ge\frac{w(z)+w(x)}2\ge\frac{w(x)}2.
\]
When $z\in\partial C$, choose $z_j\in C$ with $z_j\to z$, apply the same inequality at $(x+z_j)/2$, and pass to the limit using continuity of $w$ at the interior point $(x+z)/2$. This implies \(w((x+z)/2)\ge w(x)/2\) without assuming a boundary value for \(w\). Consequently,
\[
 \int_{\frac12E}\rho
 \ge2^{-m}\int_{E\cap C}\rho((x+z)/2)\,\dd x
 \ge2^{-(m+q)}\int_E\rho.
\]
The proof of the equality for $\rho^*$ is the same. For $q=0$, the weights are constant and the same homothety gives the factor $2^m$. This proves \eqref{qrigid-doubling}. The homeomorphism and regularity assertions \eqref{qrigid-homeomorphisms} and \eqref{qrigid-legendre} follow from Proposition~\ref{prop:localfinite} and Caffarelli's standard interior regularity theory.

We next justify the coarea formula and absolute continuity. At every interior point,
\begin{equation*}
 \nabla\Psi=T+Hx,
 \qquad x\cdot\nabla\Psi=\Psi+x^tHx\ge\Psi.
\end{equation*}
Fix \(0<h_0<h_1<h_2<\infty\). By hypothesis,
\(\{\Psi<h_2\}\subset B_R\) for some \(R<\infty\). Hence
\begin{equation}\label{qrigid-gradlower}
 |\nabla\Psi|\ge\frac{h_0}{R}
 \qquad\text{on }\{h_0\le\Psi\le h_1\}.
\end{equation}
Choose a sequence \(\chi_\ell\in C_c^\infty(C)\) such that
\[
 0\le\chi_\ell\uparrow1
 \qquad\text{pointwise in }C.
\] 
If \(N\subset[h_0,h_1]\) is null, by coarea and
\eqref{qrigid-gradlower}, we have
\[
 \int_{\Psi^{-1}(N)}\rho \chi_\ell\,\dd x
 \le\frac R{h_0}
 \int_N\int_{\{\Psi=h\}}\rho \chi_\ell\,
       d\mathcal H^{m-1}\,dh=0.
\]
Monotone convergence shows that \(\Psi_\#(\rho\,\dd x)\) is absolutely continuous on \([h_0,h_1]\). Applying coarea once more gives, for every Borel \(E\subset[h_0,h_1]\),
\[
 \int_{\{\Psi\in E\}}\rho \,\dd x
 =\int_E\int_{\{\Psi=h\}}
       \frac{\rho }{|\nabla\Psi|}\,d\mathcal H^{m-1}\,dh.
\]
Since the band is arbitrary, \(M\) is locally absolutely continuous on \((0,\infty)\), and
\[
 M'(h)=\int_{\{\Psi=h\}}\frac{\rho }{|\nabla\Psi|}
          \,d\mathcal H^{m-1}
 \qquad\text{for a.e. }h>0.
\]
The same argument applies to \(\Psi^*\).
\end{proof}

The proof of the next theorem is modeled on the monotonicity formula and its equality case in Collins--Tong \cite[Theorem~3.1]{CollinsTong2025}, using the calculation in \cite[proof of Proposition~3.1]{CollinsTong2025}. We include the details because neither homogeneity of the densities nor conic structure of the domains is assumed here. 
\begin{theorem}
\label{qrigidi}
Let \(m\ge1\), \(q\ge0\), and put \(n=m+q\). Let \(\rho\,\dd x\) and \(\rho^*\,\dd y\) be locally finite measures on convex domains \(C,C^*\subset\mathbb R^m\). If \(q>0\), assume that \(\rho^{1/q}\) and \((\rho^*)^{1/q}\) are nonnegative and concave. If \(q=0\), assume that \(\rho,\rho^*\) are positive constants. Let \(\phi,\varphi\) be dual convex potentials satisfying \eqref{qrigid-extensions}, and assume that \(D\phi\) transports \(\rho\,\dd x\) to \(\rho^*\,\dd y\). Assume that \(C\) and \(C^*\) are the interiors of the supports of $\rho, \rho^*,$ respectively, and \(0\in\overline C\cap\overline{C^*}\). Assume also that
\begin{equation}\label{qrigid-center}
  \phi(0)+ \varphi (0)=0,
 \qquad 0\in\partial \phi(0),
 \qquad 0\in\partial \varphi (0).
\end{equation}
Suppose that, for every \(h>0\), the sets
\[
 \{x\in C:x\cdot D\phi(x)<h\},
 \qquad
 \{y\in C^*:y\cdot D\varphi(y)<h\}
\]
are bounded and have finite mass. If
\begin{equation*}
 r^{-n}\int_{\{x\cdot D\phi(x)<r^2\}}\rho(x)\,\dd x
\end{equation*}
is a positive constant for \(0<r<\infty\), then \(C,C^*\) are cones, \(\rho,\rho^*\) are homogeneous of degree \(q\), and the normalized potentials $\phi- \phi(0),\ 
 \varphi- \varphi (0)$ are homogeneous of degree two.
\end{theorem}

\begin{proof}
After replacing \(\phi\) by \(\phi-\phi(0)\) and \(\varphi\) by \(\varphi-\varphi(0)\), we may assume that
\begin{equation}\label{qrigid-normalization}
  \phi(0)= \varphi (0)=0,\ \ 
 0\in\partial \phi(0),\quad
 0\in\partial \varphi (0).
\end{equation}
Set
\[
 T=D\phi,\qquad S=D\varphi,\qquad
 \Psi(x)=x\cdot T(x),\qquad \Psi^*(y)=y\cdot S(y),
\]
and
\[
 K_h=\{x\in C:\Psi(x)<h\},\qquad
 L_h=\{y\in C^*:\Psi^*(y)<h\},\qquad
 M(h)=\int_{K_h}\rho.
\]
By Lemma~\ref{qrigid-local-calculus},
\[
 T(K_h)=L_h,\qquad M(h)=\int_{L_h}\rho^*(y)\,\dd y,
\]
and all the regularity and coarea formulas used below hold. The mass assumption is exactly
\begin{equation}\label{Mhflat}
 M(h)=ch^{n/2}\qquad(h>0).
\end{equation}
For almost every \(h>0\), define
\[
 A_C(h)=\int_{\{\Psi=h\}}\rho
 \frac{x\cdot\nabla\Psi}{|\nabla\Psi|}
 \,\dd\mathcal H^{m-1}
\]
and define \(A_{C^*}(h)\) analogously.

\smallskip
\noindent\emph{Step 1: estimates for \(A_C\) and \(A_{C^*}\).} We claim that
\begin{equation}\label{Acbounds}
 A_C(h)\le nM(h),\qquad A_{C^*}(h)\le nM(h).
\end{equation}
It suffices to prove the first inequality. For \(\theta\in\mathbb S^{m-1}\), write
\[
 C\cap\mathbb R_+\theta=(0,R_C(\theta))\theta,\qquad
 K_h\cap\mathbb R_+\theta=(0,R_h(\theta))\theta,
\]
and let
\[
    \Gamma_C:=\{\theta\in\mathbb S^{m-1}:R_C(\theta)>0\}.
\]
Convexity and \eqref{qrigid-normalization} imply that
\[
 r\longmapsto\Psi(r\theta)
 =r\,\partial_r\phi(r\theta)
\] is nondecreasing. By Lemma~\ref{radialq}, $\rho_\theta(r):=r^{-q}\rho(r\theta)$ is nonincreasing. Since \(\partial_r\Psi>0\) on every positive level set, \(\{\Psi=h\}\) is the radial graph \(x=R_h(\theta)\theta\) over the set \(\{\theta:R_h(\theta)<R_C(\theta)\}\). The area formula for radial graphs gives
\[
 A_C(h)
 =
 \int_{\{\theta:R_h(\theta)<R_C(\theta)\}}
 \rho\bigl(R_h(\theta)\theta\bigr)
 R_h(\theta)^m\,\dd\theta.
\]
Hence
\begin{align*}
 nM(h)-A_C(h)
 &=
 n\int_{\{R_h<R_C\}}\int_0^{R_h}
 \bigl(\rho_\theta(r)-\rho_\theta(R_h)\bigr)r^{n-1}
 \,\dd r\,\dd\theta \\
 &\quad
 +n\int_{\{R_h=R_C\}}\int_0^{R_C}
 \rho_\theta(r)r^{n-1}\,\dd r\,\dd\theta .
\end{align*}
This proves \eqref{Acbounds}. Since \(\rho>0\) on \(C\), equality holds only if the set of directions for which \(R_h=R_C\) has spherical measure zero and \(\rho_\theta\) is constant on \((0,R_h)\) for almost every remaining direction.

\smallskip
\noindent\emph{Step 2: equality in the two estimates.}
Let \(y=T(x)\) and \(H=D^2\phi(x)\). Then
\[
 D^2\phi(x)>0,\qquad D^2\varphi(y)=\left(D^2\phi(x)\right)^{-1},\qquad
 \Psi^*(T(x))=\Psi(x),
\]
and
\[
 \nabla\Psi=T+D^2\phi(x)x,\qquad
 \nabla\Psi^*(T(x))=x+\left(D^2\phi(x)\right)^{-1}T(x).
\]
Fix \( 0\le \eta\in C_c((0,\infty))\). Choose increasing sequences
\[
 \alpha_\ell\in C_c^\infty(C),
 \qquad
 \beta_\ell\in C_c^\infty(C^*),
 \qquad
 0\le\alpha_\ell\rightarrow1,
 \quad
 0\le\beta_\ell\rightarrow1,
\]
and set
\[
 \gamma_\ell(x)=\alpha_\ell(x)\beta_\ell(T(x)),\qquad
 \gamma_\ell^*(y)=\alpha_\ell(S(y))\beta_\ell(y).
\]
Then $\gamma_\ell^*(T(x))=\gamma_\ell(x).$ By coarea, the transport identity, and monotone convergence, we have
\begin{align*}
 \int_0^\infty\eta(h)A_{C^*}(h)\,\dd h
 &=\lim_{\ell\to\infty}
   \int_{C^*}\eta(\Psi^*)\,\rho^*\,\gamma_\ell^*\,
        y\cdot\nabla\Psi^*\,\dd y\\
 &=\lim_{\ell\to\infty}
   \int_C\eta(\Psi)\,\rho\,\gamma_\ell\,
        \bigl(\Psi+T^t\left(D^2\phi(x)\right)^{-1}T\bigr)\,\dd x\\
 &=\lim_{\ell\to\infty}\int_0^\infty\eta(h)
   \int_{\{\Psi=h\}}\frac{\rho\gamma_\ell}{|\nabla\Psi|}
        \bigl(h+T^t\left(D^2\phi(x)\right)^{-1}T\bigr)
        \,\dd\mathcal H^{m-1}\,\dd h\\
 &=\int_0^\infty\eta(h)
   \int_{\{\Psi=h\}}\frac{\rho}{|\nabla\Psi|}
        \bigl(h+T^t\left(D^2\phi(x)\right)^{-1}T\bigr)
        \,\dd\mathcal H^{m-1}\,\dd h .
\end{align*}
Since \(\eta\) is arbitrary, for almost every \(h>0\),
\begin{align*}
 A_C(h) &=\int_{\{\Psi=h\}} \frac{\rho}{|\nabla\Psi|} \bigl(h+x^tD^2\phi(x)x\bigr) \,\dd\mathcal H^{m-1},
 \\
 A_{C^*}(h) &= \int_{\{\Psi=h\}} \frac{\rho}{|\nabla\Psi|} \bigl(h+T^t\left(D^2\phi(x)\right)^{-1}T\bigr) \,\dd\mathcal H^{m-1}.
\end{align*}
Since \(h=\Psi=x\cdot T\) on \(\{\Psi=h\}\),
\begin{align*}
 A_C(h)+A_{C^*}(h)-4hM'(h)
 &=
 \int_{\{\Psi=h\}}
 \frac{\rho}{|\nabla\Psi|}
 \left|\left(D^2\phi(x)\right)^{1/2}x-\left(D^2\phi(x)\right)^{-1/2}T\right|^2
 \,\dd\mathcal H^{m-1}
 \\
 &=:\mathcal E(h)\ge0.
\end{align*}
By \eqref{Mhflat}, $4hM'(h)=2nM(h).$
Combining this identity with \eqref{Acbounds} gives
\[
 2nM(h)\ge A_C(h)+A_{C^*}(h)
 =2nM(h)+\mathcal E(h)\ge 2nM(h).
\]
Therefore
\begin{equation}\label{qrigid-zerodefects}
 A_C(h)=A_{C^*}(h)=nM(h),\qquad \mathcal E(h)=0
 \quad\text{for a.e. }h>0.
\end{equation}

\smallskip
\noindent\emph{Step 3: homogeneity of the densities and potentials.}
Choose \(h_j\rightarrow\infty\) for which \eqref{qrigid-zerodefects} holds. Outside a fixed null set of directions, the equality case in Step~1 gives
\[
 R_{h_j}<R_C
 \quad\text{and}\quad
 \rho_\theta\ \text{constant on }(0,R_{h_j}(\theta))
\]
for every \(j\). Since \(K_{h_j}\rightarrow C\), $R_{h_j}(\theta)\rightarrow R_C(\theta).$ It follows that
\begin{equation}\label{qrigid-homodensity}
    \rho(tx)=t^q \rho(x), \qquad \rho^*(ty)=t^q \rho^*(y),
\end{equation}
whenever \(x,tx\in C\) and \(y,ty\in C^*\).

By strict convexity and \eqref{qrigid-normalization}, we have
\[
 \Psi>0\quad\text{in }C\setminus\{0\}.
\]
For every \(I\Subset(0,\infty)\), by coarea formula and \eqref{qrigid-zerodefects}, we have
\[
 0=\int_I\mathcal E(h)\,\dd h=\int_{\{\Psi\in I\}} \rho(x)
 \left|\left(D^2\phi(x)\right)^{1/2}x-\left(D^2\phi(x)\right)^{-1/2}T(x)\right|^2 \,\dd x.
\]
Exhausting \(\{\Psi>0\}\) and using continuity gives $D^2\phi(x)x=D\phi(x)\ (x\in C).$ Consequently,
\[
     D\bigl(x\cdot D\phi-2\phi\bigr)=0,
\]
so $x\cdot D\phi(x)-2\phi(x)=c_0 \ (x\in C)$ for some constant \(c_0\). Fix \(x\in C\) and set \(w(t)=\phi(tx)\). On every interval on which \(tx\in C\),
\[
 tw'(t)-2w(t)=c_0,
\]
and hence $w(t)=At^2-\frac{c_0}{2}.$ By convexity and \eqref{qrigid-normalization},
\[
 0\le\phi(tx)\le t\phi(x)\longrightarrow 0
 \qquad\text{as }t\rightarrow  0.
\]
Thus \(c_0=0\), and
\begin{equation}\label{qrigid-homopotential}
    \phi(tx)=t^2\phi(x)
\end{equation}
whenever \(x,tx\in C\). The same argument applies to \(\varphi\).

\smallskip
\noindent\emph{Step 4: the domains are cones.}
Recall that
\[
 K_h=\{x\in C:\Psi(x)<h\}.
\]
Fix \(r>1\). Since \(C\) is convex and \(0\in\overline C\), we have
\[
 r^{-1}C\subset C.
\]
The homogeneity identities already proved imply
\[
 K_{r^2}
 =r\bigl(K_1\cap r^{-1}C\bigr).
\]
Indeed, if \(z\in K_1\cap r^{-1}C\), then \(rz\in C\) and
\[
 \Psi(rz)=r^2\Psi(z)<r^2.
\]
The converse follows by applying the same identity to \(z=x/r\).

Using \(\rho(rz)=r^q\rho(z)\) and \(n=m+q\), we obtain
\[
\begin{aligned}
 M(r^2)
 &=\int_{K_{r^2}}\rho(x)\,\dd x\\
 &=r^m\int_{K_1\cap r^{-1}C}\rho(rz)\,\dd z\\
 &=r^n\int_{K_1\cap r^{-1}C}\rho(z)\,\dd z.
\end{aligned}
\]
Consequently,
\begin{equation}\label{monoeq}
 r^{-n}M(r^2)
 =
 \int_{K_1\cap r^{-1}C}\rho(z)\,\dd z.
\end{equation}
By \eqref{Mhflat}, the left-hand side of \eqref{monoeq} equals
\[
 M(1)=\int_{K_1}\rho(z)\,\dd z.
\]
Since \(K_1\cap r^{-1}C\subset K_1\) and \(\rho>0\) in \(C\), it follows that
\[
 K_1\subset r^{-1}C.
\]
Equivalently,
\begin{equation}\label{K1expands}
 rK_1\subset C
 \qquad(r>1).
\end{equation}

Now fix \(x\in C\) and \(t>1\). Since \(0\in\overline C\), we have \(sx\in C\) for every \(0<s<1\). Moreover,
\[
 \Psi(sx)=s^2\Psi(x)\longrightarrow0
 \qquad(s\rightarrow 0).
\]
Thus we may choose \(s\in(0,1)\) such that \(sx\in K_1\). Applying \eqref{K1expands} with \(r=t/s>1\), we obtain
\[
 tx=\frac ts(sx)\in C.
\]
For \(0<t<1\), the inclusion \(tx\in C\) follows directly from the convexity of \(C\) and \(0\in\overline C\). Hence
\[
 tC=C\qquad(t>0),
\]
so \(C\) is a cone. The same argument applies to \(C^*\). The homogeneity identities for \(\rho,\rho^*,\phi,\varphi\) are therefore global.
\end{proof}

We now identify the normalized limit.
\begin{theorem}\label{sequential}
Under \eqref{john}--\eqref{Fflat}, the centered sub-level sets at heights $1$ and $1/4$ subconverge, and their limits satisfy
\[
 S_{1/4}^c[\tilde u_\infty](0)
 =\tfrac12S_1^c[\tilde u_\infty](0).
\]
\end{theorem}

\begin{proof}
Apply Lemma~\ref{diacompac}. After passing to a subsequence, we obtain a locally uniform limit \(\tilde u_\infty\), its conjugate \(\tilde v_\infty=\tilde u_\infty^*\), the compatible limiting plan \(\pi\), and the Hausdorff limits of the two centered sub-level sets. Let \(\mu\) be the first marginal of \(\pi\), and set
\[
 L=\operatorname{aff}(\spt\mu),\qquad m=\dim L.
\]
By Lemma~\ref{gmarconc}, \(L\) is a linear subspace containing the origin and \(1\le m\le n\).

Suppose first that \(m=n\). Then $\mu=\rho_{\infty}\mathbf1_C\,\dd x$ for an open convex set \(C\) and a constant \(\rho_{\infty}>0\). The argument in the proof of Lemma~\ref{fibredu}, from
\eqref{dualKj} through the subsequent identification of the full second marginal, does not use \(m<n\). Therefore, we have
\[
 (\operatorname{pr}_2)_\#\pi=\rho^*_{\infty}\mathbf1_{C^*}\,\dd y,
 \qquad
 C^*=\operatorname{int}(\operatorname{dom}\tilde v_\infty),
\]
for some constant \(\rho^*_{\infty}>0\). Let \(\phi=\tilde u_{\infty}, \varphi=(\tilde u_\infty+I_C)^*\). For \(\pi\)-almost every \((x,y)\), we have \(x\in C\), \(y\in C^*\), and
\[
 \tilde u_\infty(x)+\tilde v_\infty(y)=x\cdot y.
\]
Consequently,
\[
 \varphi(y)
 \ge x\cdot y-\tilde u_\infty(x)
 =\tilde v_\infty(y).
\]
Since \(\varphi\le\tilde v_\infty\) by definition, the density of the second projection of \(\pi\) and continuity imply \(\varphi=\tilde v_\infty\) on \(C^*\). Hence by \eqref{dualidentification} and the definition of \(\varphi\), we have
\[
 \phi(x)=\sup_{y\in C^*}\{x\cdot y-\varphi(y)\},
 \qquad
 \varphi(y)=\sup_{x\in C}\{x\cdot y-\phi(x)\}.
\]
Thus \(\phi\) and \(\varphi\) are Legendre dual relative to \(C\) and \(C^*\). Proposition~\ref{prop:localfinite}, applied in both directions, shows that
\[
 D\phi:C\longrightarrow C^*,
 \qquad
 D\varphi:C^*\longrightarrow C
\]
are homeomorphisms. Lemma~\ref{diacompac}(iii) and
Proposition~\ref{noescape} show that the projections of
\[
 \spt\bigl(\pi\lfloor\{(x,y):x\cdot y<r^2\}\bigr)
\]
onto the two factors are bounded. Since \(D\phi,D\varphi\) are continuous and the limiting densities are positive constants, it follows that, for every \(r>0\), the sets
\[
 \{x\in C:x\cdot D\phi(x)<r^2\},
 \qquad
 \{y\in C^*:y\cdot D\varphi(y)<r^2\}
\]
are bounded. Moreover, Lemma~\ref{diacompac}(iii) gives
\[
 \mu\{x\in C:x\cdot D\phi(x)<r^2\}
 =\pi\{(x,y):x\cdot y<r^2\}
 =r^n,
\]
and hence
\[
 r^{-n}\mu\{x\in C:x\cdot D\phi(x)<r^2\}=1
 \qquad(r>0).
\]
Thus the limit is precisely the constant-density equality case of the Collins--Tong monotonicity formula. Since \(\phi(0)=\varphi(0)=0\) and \(\phi,\varphi\ge0\), the normalization gives \eqref{qrigid-center}. Hence all the hypotheses of Theorem~\ref{qrigidi} with \(q=0\) are satisfied. Its \(q=0\) argument is precisely the constant-density rigidity of Collins--Tong \cite[Theorem~3.1]{CollinsTong2025}, as used in \cite[proof of Theorem~4.1]{CollinsTong2025}. Hence \(C,C^*\) are cones and \(\phi=\tilde u_\infty,\varphi\) are 2-homogeneous.

It remains to consider the case \(m<n\). By Lemma \ref{gmarconc}, we have \(1\le m<n\) and, for every fixed \(R<\infty\),
\[
\sup_{x\in D_R(\tilde u_j,0)}
\operatorname{dist}(x,L)\longrightarrow 0.
\]
Hence the hypotheses of Lemma \ref{fibredu} are satisfied.

Put $q=n-m$. Lemma \ref{fibredu} gives a full-dimensional dual domain \(\Omega_{\infty}^*\), dual convex potentials \(\phi\) on \(C\) and \(\varphi\) on \(C^*:=P_L\Omega_{\infty}^*\), and a constant \(\rho^*_{\infty}>0\) such that
\begin{equation}\label{factorV}
 \tilde v_{\infty}(y)=\varphi(P_Ly),\qquad y\in \Omega_{\infty}^*.
\end{equation}
Moreover, with
\[
 \rho^*_L(\xi)=\rho^*_{\infty}\,\mathcal H^q(\Omega_\infty^*\cap(\xi+L^\perp)),
\]
every fiber has finite positive volume, \((\rho^*_L)^{1/q}\) is positive and concave, and
\begin{equation}\label{reducedflat}
 (D\varphi)_\#(\rho^*_L\,\dd\xi)
   =\rho\mathbf1_C\,d\mathcal H^m\lfloor L,
 \qquad
 r^{-n}\int_{\{\xi\cdot D\varphi(\xi)<r^2\}}\rho^*_L(\xi)\,\dd\xi=1
 \quad(r>0).
\end{equation}
By \eqref{dualidentification} and \eqref{factorV},
\begin{equation}\label{reduceddual}
 \phi(x)=\tilde u_{\infty}(x)
 =\sup_{y\in \Omega_\infty^*}\{x\cdot y-\tilde v_\infty(y)\}
 =\sup_{\xi\in C^*}\{x\cdot\xi-\varphi(\xi)\},
 \qquad x\in L.
\end{equation}

Define
\[
  \underline\phi:=(\varphi+I_{C^*})^*,
 \qquad
 \underline\varphi:=(\phi+I_C)^*.
\]

By \eqref{reduceddual},
\[
 \underline\phi=\phi\quad\text{on }L,
 \qquad
 x\cdot\xi-\phi(x)\le\varphi(\xi)
 \quad(x\in C,\ \xi\in C^*).
\]
Hence
\[
 \underline\varphi(\xi)
 =\sup_{x\in C}\{x\cdot\xi-\phi(x)\}
 \le\varphi(\xi)
 \qquad(\xi\in C^*).
\]

On the other hand, for \(\rho_L^*\,\dd\xi\)-almost every \(\xi\in C^*\), \eqref{reducedflat} implies \(x=D\varphi(\xi)\in C\). By Legendre duality we have 
\[
 \phi(x)+\varphi(\xi)=x\cdot\xi,
\]
and therefore
\[
 \underline\varphi(\xi)
 \ge x\cdot\xi-\phi(x)
 =\varphi(\xi).
\]
Thus \(\underline\varphi=\varphi\) almost everywhere in \(C^*\). Since both functions are finite and convex in \(C^*\), they are continuous, and hence
\[
 \underline\varphi=\varphi\qquad\text{in }C^*.
\]
Therefore \(\underline\phi\) and \(\underline\varphi\) are the minimal convex extensions of \(\phi\) and \(\varphi\), respectively, and satisfy \eqref{qrigid-extensions}.

Moreover,
\[
 \underline\phi^*
 =\operatorname{cl}(\varphi+I_{C^*}),
 \qquad
 \underline\varphi^*
 =\operatorname{cl}(\phi+I_C).
\]
Consequently,
\[
 \partial\underline\phi(\mathbb R^m)\subset\overline{C^*},
 \qquad
 \partial\underline\varphi(\mathbb R^m)\subset\overline C.
\]


The normalization and the continuity of \(\tilde u_{\infty}\) provide a sequence \(x_\ell\in C\) such that
\[
 x_\ell\to0,\qquad \phi(x_\ell)\to0.
\]
Moreover, \eqref{dualidentification} and \(\tilde v_{\infty}(0)=0\) provide \(y_\ell\in \Omega^*_{\infty}\) such that
\[
 y_\ell\to0,\qquad \tilde v_{\infty}(y_\ell)\to0.
\]
Setting \(\xi_\ell=P_Ly_\ell\), \eqref{factorV} gives
\[
 \xi_\ell\to0,\qquad \varphi(\xi_\ell)\to0.
\]
Consequently the two minimal extensions are nonnegative, vanish at the origin, and satisfy
\[
  \underline\phi(0)=\underline{\varphi}(0)=0,\qquad
 0\in\partial \underline\phi(0),\qquad
 0\in\partial\underline{\varphi}(0),
\]
so \eqref{qrigid-center} holds.

Since \(\rho^{1/q}\) and \((\rho^*_L)^{1/q}\) are concave, by the proof of \eqref{qrigid-doubling} in Lemma \ref{qrigid-local-calculus} we have that $\rho\dd \mathcal{H}^m$ and $\rho^*_L\dd \xi$ are locally doubling. Proposition~\ref{prop:localfinite} shows that \(D\phi\) and \(D\varphi\) are homeomorphisms.
Let
\[
 \widehat\pi:=(\operatorname{id},P_L)_\#\pi,
 \qquad
 \widehat G_r:=\{(x,\xi)\in L\times L:x\cdot\xi<r^2\}.
\]
Since the first marginal of \(\pi\) is supported on \(L\), we have \(x\cdot P_Ly=x\cdot y\) for \(\pi\)-almost every \((x,y)\), and hence
\[
 \widehat\pi\lfloor\widehat G_r
 =(\operatorname{id},P_L)_\#(\pi\lfloor G_r).
\]
For fixed \(r>0\), Proposition~\ref{noescape} implies the supports of \(\pi_j\lfloor G_r\) in a common compact set. Passing to the limit in \eqref{limitplans} shows that \(\pi\lfloor G_r\) is supported in the same compact set. Therefore,
\[
 \widehat\pi\lfloor\widehat G_r
 =(\operatorname{id},P_L)_\#(\pi\lfloor G_r)
\]
has compact support.

Set
\[
 A_r:=\{x\in C:x\cdot D\phi(x)<r^2\},
 \qquad
 A'_r:=\{\xi\in C^*:\xi\cdot D\varphi(\xi)<r^2\}.
\]
Using
\[
 \widehat\pi
 =(\operatorname{id},D\phi)_\#
   \bigl(\rho\mathbf1_C\,\dd\mathcal H^m\bigr)
 =(D\varphi,\operatorname{id})_\#
   \bigl(\rho_L^*\mathbf1_{C^*}\,\dd\xi\bigr),
\]
we obtain
\[
\begin{aligned}
 (\operatorname{pr}_1)_\#
   (\widehat\pi\lfloor\widehat G_r)
 &=\rho\mathbf1_{A_r}\,\dd\mathcal H^m,\\
 (\operatorname{pr}_2)_\#
   (\widehat\pi\lfloor\widehat G_r)
 &=\rho_L^*\mathbf1_{A'_r}\,\dd\xi.
\end{aligned}
\]
Since \(D\phi\) and \(D\varphi\) are continuous, the functions
\[
 x\longmapsto x\cdot D\phi(x),
 \qquad
 \xi\longmapsto\xi\cdot D\varphi(\xi)
\]
are continuous. Hence \(A_r\subset C\) and \(A'_r\subset C^*\) are open. Since \(\rho>0\) in \(C\) and
\(\rho_L^*>0\) in \(C^*\),
\[
 \operatorname{spt}
 \bigl(\rho\mathbf1_{A_r}\,\dd\mathcal H^m\bigr)
 =\overline{A_r},
 \qquad
 \operatorname{spt}
 \bigl(\rho_L^*\mathbf1_{A'_r}\,\dd\xi\bigr)
 =\overline{A'_r},
\]
where the closures are taken in \(L\). Both supports are compact, being the coordinate projections of \(\operatorname{spt}(\widehat\pi\lfloor\widehat G_r)\). Thus \(A_r\) and \(A'_r\) are bounded.

Moreover, \eqref{reducedflat} and the transport identity give
 $\int_{A_r}\rho\,\dd\mathcal H^m
 =\int_{A'_r}\rho^*_L\,\dd \xi
 =r^n.$ All the hypotheses of Theorem~\ref{qrigidi} are now satisfied. Hence \(C,C^*\) are cones, \(\phi,\varphi\) are two-homogeneous, and \(\rho,\rho^*_L\) are \(q\)-homogeneous.

It remains to lift the homogeneity from the projected problem. Denote
\[
 \Omega^*_{\infty,\xi}:=\Omega_\infty^*\cap(\xi+L^\perp).
\]
Since \(\Omega_\infty^*\) is convex and \(0\in\overline{\Omega_\infty^*}\),
\[
 t\Omega^*_{\infty,\xi}\subset \Omega^*_{\infty,t\xi}\qquad(0<t<1).
\]
The \(q\)-homogeneity of \(\rho^*_L\) gives $|\Omega^*_{\infty,t\xi}|=t^q|\Omega^*_{\infty,\xi}|.$ The fibers are open convex sets of finite positive volume. Hence, we have
\[
 \Omega^*_{\infty,t\xi}=t\Omega^*_{\infty,\xi}\qquad(0<t<1).
\]
Since \(C^*\) is a cone, applying this identity at \(s\xi\) with \(t=s^{-1}\) gives the same equality for \(s>1\). Therefore \(\Omega^*_{\infty}\) is a cone. Together with \eqref{factorV}, the fact that \(\Omega_\infty^*\) is a cone shows that
\[
 \operatorname{cl}\bigl(
 (\tilde v_\infty|_{\Omega_\infty^*})
 +I_{\Omega_\infty^*}\bigr)
\]
is two-homogeneous. By \eqref{dualidentification}, this function equals \(\tilde v_\infty\), while \(\tilde u_\infty=\tilde v_\infty^*\). Hence both \(\tilde v_\infty\) and \(\tilde u_\infty\) are two-homogeneous.

Thus \(\tilde u_\infty\) is two-homogeneous in both alternatives. If \(p_1\) is the centering vector at height \(1\), then
\[
 \sqrt h\,S_1^c[\tilde u_\infty](0)
 =\{x:\tilde u_\infty(x)-\sqrt h\,p_1\cdot x<h\}.
\]
This set has center of mass zero. By uniqueness of the centering plane,
\[
 S^c_h[\tilde u_\infty](0)=\sqrt h\,S^c_1[\tilde u_\infty](0).
\]
Taking \(h=1/4\), and using Lemma~\ref{diacompac}(ii), proves the assertion.
\end{proof}

\section{Global \texorpdfstring{$W^{2,p}$}{W{2,p}} estimates}
\label{sec7shape}

We first record how density oscillations behave under the blow-up normalization. This also shows that one-sided affine normalization preserves the small relative oscillation on the target sub-level sets.

\begin{proposition}
\label{prop:continuous-density-interface}
Suppose $u, v, f, g$ satisfy \eqref{breniereq} and \eqref{eqdata}. Let \(r_j\rightarrow0\), \(x_j\in\overline\Omega\), and \(y_j=D u(x_j)\). For each \(j\), apply the normalization of
Section~\ref{secblowup} at \((x_j,y_j)\) with height \(h_j=r_j^2\), choosing the positive definite linear map \(A_j\) so that
\[
 A_j(S_{r_j^2}^c[u](x_j)-x_j)
\]
is in John position. Denote the resulting potentials, densities, and measures by \(\tilde u_j,\tilde v_j,\rho_j,\rho_j^*,\mu_j,\nu_j\), and choose the common mass factor so that
\[
 \mu_j(D_1(\tilde u_j,0))=1.
\]
Then, for every fixed \(R<\infty\),
\[
 \operatorname{osc}^{\rm rel}_{D_R(\tilde u_j,0)}\rho_j
 +\operatorname{osc}^{\rm rel}_{D_R(\tilde v_j,0)}\rho_j^*
 \longrightarrow0.
\]
Moreover, \(\mu_j\) and \(\nu_j\) have a common doubling constant depending only on \(n\) and \(\Lambda/\lambda\).
\end{proposition}

\begin{proof}
By the change of variable in Section~\ref{secblowup}, the multiplicative normalization factors do not affect relative oscillations, and
\begin{align*}
 D_R(\tilde u_j,0)
 &=A_j\bigl(D_{r_jR}(u,x_j)-x_j\bigr),\\
 D_R(\tilde v_j,0)
 &=r_j^{-2}A_j^{-t}
  \bigl(D_{r_jR}(v,y_j)-y_j\bigr).
\end{align*}

Let \(\omega_f\) and \(\omega_g\) be the moduli of continuity of \(f\) and \(g\). By \eqref{diameterup}, for every fixed \(R\) and all sufficiently large \(j\),
\begin{align*}
 \operatorname{osc}^{\rm rel}_{D_R(\tilde u_j,0)}\rho_j
 &\le \lambda^{-1}\omega_f\bigl(C(r_jR)^\sigma\bigr),\\
 \operatorname{osc}^{\rm rel}_{D_R(\tilde v_j,0)}\rho_j^*
 &\le \lambda^{-1}\omega_g\bigl(C(r_jR)^\sigma\bigr).
\end{align*}
Both right-hand sides tend to zero as \(j\to\infty\).

Finally, the original measures have a common doubling constant depending only on \(n\) and \(\Lambda/\lambda\). Affine changes of variables and multiplication by positive constants preserve the doubling property, which proves the last assertion.
\end{proof}

We now return to the original continuous densities. All estimates below are uniform in the base point \(x_0\).

\begin{proposition}
\label{prop:zero-average-final}
There are \(C,c_0>0\) such that
\begin{equation}\label{Funiform}
 F_{x_0}(r)\le C\quad(0<r\le r_0),
 \qquad F_{x_0}(r_0)\ge c_0
\end{equation}
for every base point \(x_0\).  For \(0<a<b\le r_0\), put
\[
 \mathfrak a_{x_0}(a,b)
 =\log\frac{F_{x_0}(a)}{F_{x_0}(b)}
  +C\int_a^b\frac{\vartheta(s)}s\,\dd s.
\]
This quantity is nonnegative and additive on adjacent intervals. Let \(r_k=2^{-k}r_0\) and
\[
 a_k=\mathfrak a_{x_0}(r_{k+1},r_k).
\]
Then
\begin{equation}\label{zeroaverage}
 \sup_{x_0}\sum_{k=0}^{N-1}a_k=o(N).
\end{equation}
\end{proposition}

\begin{proof}
By \eqref{include-D-S}, \eqref{include-S-Sc} and \eqref{volumeproduct} we have
\[
 |D_r(u,x_0)|\,|D_r(v,y_0)|\le C r^{2n}.
\]
Since \(\lambda\leq f, g\leq \Lambda\), we have \(M(r^2)^2\le C r^{2n}\), proving the first part of \eqref{Funiform}. The second part follows easily from  the \(C^{1,\alpha_0}\) estimate of $u$ in Proposition~\ref{classicalproperties}.

By \eqref{endpointinequality} we have that \(\mathfrak a_{x_0}(a,b)\ge0\), and additivity is obvious. A straightforward computation shows
\[
 \sum_{k=0}^{N-1}a_k
 =\log\frac{F_{x_0}(r_N)}{F_{x_0}(r_0)}
  +C\int_{r_N}^{r_0}\frac{\vartheta(s)}s\,\dd s.
\]
The left side is nonnegative, while \eqref{Funiform} bounds the logarithm from above uniformly. Since \(\vartheta(s)\to0\) as \(s\rightarrow0\) and \(r_N=2^{-N}r_0\), splitting the integral over the dyadic intervals \([r_{\ell+1},r_\ell]\) gives
\[
 \frac1N\int_{r_N}^{r_0}\frac{\vartheta(s)}s\,\dd s
 \longrightarrow0.
\]
Equivalently,
\[
 \int_{r_N}^{r_0}\frac{\vartheta(s)}s\,\dd s=o(N).
\]
This proves \eqref{zeroaverage}.
\end{proof}

\begin{proposition}
\label{prop:subpower-final}
For every \(\varepsilon>0\) there are \(C_\varepsilon,h_\varepsilon>0\), independent of the center, such that
\begin{equation}\label{shapeestimate}
 B_{C_\varepsilon^{-1}h^{1/2+\varepsilon}}(x_0)
 \subset S_h^c[u](x_0)
 \subset B_{C_\varepsilon h^{1/2-\varepsilon}}(x_0),
 \qquad0<h<h_\varepsilon.
\end{equation}
The same estimate holds for \(v\).
\end{proposition}
\begin{proof}
Fix \(\varepsilon>0\), and choose \(\tau\in(0,1/2)\), to be specified below. Recall that \(r_k=2^{-k}r_0\), and put
\[
 K_k(x_0):=S_{r_k^2}^c[u](x_0)-x_0.
\]
For integers \(Q\ge1\) and \(k\ge Q\), set
\[
 \Delta_{k,Q}(x_0)
 :=\mathfrak a_{x_0}(r_{k+Q},r_{k-Q}).
\]
We claim that there exist \(Q\in\mathbb N\), \(\delta>0\), and \(k_*\in\mathbb N\), all independent of \(x_0\), such that
\(k\ge\max\{Q,k_*\}\) and \(\Delta_{k,Q}(x_0)\le\delta\) imply
\begin{equation}\label{goodKk}
 (1-\tau)\tfrac12K_k(x_0)
 \subset K_{k+1}(x_0)
 \subset(1+\tau)\tfrac12K_k(x_0).
\end{equation}

Suppose not. For \(Q=j\), \(\delta=j^{-1}\), and \(k_*=2j\) , there exist \(x_j\in\overline\Omega\) and \(k_j\ge2j\) such that
\begin{equation}\label{badblowup}
 \Delta_{k_j,j}(x_j)\le j^{-1},
 \qquad
 \eqref{goodKk}\text{ fails at }(x_j,k_j).
\end{equation}
Let \((U_j,V_j,\mu_j,\nu_j)\) be the normalized blow-up of the original problem at \((x_j,Du(x_j))\) and height \(r_{k_j}^2\), with affine transform \(A_j\) and common mass factor chosen so that \(\mu_j(D_1(U_j,0))=1\). By affine equivariance of centered sections, we have
\begin{equation}\label{sectionrescaling}
 S_s^c[U_j](0)
 =A_j\bigl(S_{s r_{k_j}^2}^c[u](x_j)-x_j\bigr)
 \qquad(s>0).
\end{equation}
In particular, the sections at heights \(1\) and \(1/4\) are the images of \(K_{k_j}(x_j)\) and \(K_{k_j+1}(x_j)\), respectively.

The normalized mass function satisfies
\[
 \widehat F_j(s):=s^{-n}\mu_j(D_s(U_j,0))
 =\frac{F_{x_j}(r_{k_j}s)}{F_{x_j}(r_{k_j})}.
\]
Fix \(R>1\). For all sufficiently large \(j\), we have \(R\le2^j\). By the nonnegativity and additivity of \(\mathfrak a_{x_j}\), \eqref{badblowup} implies, for \(R^{-1}\le s\le R\),
\[
 \mathfrak a_{x_j}\!\left(
  \min\{r_{k_j}s,r_{k_j}\},
  \max\{r_{k_j}s,r_{k_j}\}
 \right)\le j^{-1}.
\]
Consequently,
\[
 \sup_{R^{-1}\le s\le R}|\log\widehat F_j(s)|
 \le j^{-1}
  +C(\log R)\sup_{0<t\le Rr_{k_j}}\vartheta(t)
 \longrightarrow0,
\]
because \(k_j\ge2j\) and hence \(r_{k_j}\to0\). Thus
\(\widehat F_j\to1\) locally uniformly on \((0,\infty)\), which is
\eqref{Fflat}.

Let \(\rho_j,\rho_j^*\) be the densities of \(\mu_j,\nu_j\). Since
\(k_j\ge2j\),
\[
 r_{k_j}2^j=r_{k_j-j}\le r_j\longrightarrow0.
\]
By the estimates in the proof of Proposition~\ref{prop:continuous-density-interface}, we have that
\begin{align*}
 \operatorname{osc}^{\rm rel}_{D_{2^j}(U_j,0)}\rho_j
 +\operatorname{osc}^{\rm rel}_{D_{2^j}(V_j,0)}\rho_j^*
 &\le \lambda^{-1}\omega_f\bigl(C(r_{k_j}2^j)^\sigma\bigr)
     +\lambda^{-1}\omega_g\bigl(C(r_{k_j}2^j)^\sigma\bigr)
 \longrightarrow0.
\end{align*}
This implies \eqref{relativeosc} at every fixed radius. Moreover, these actual blow-ups satisfy \eqref{datanormalized}, \eqref{densityratio}, \eqref{john}, and \eqref{massone}, and Proposition~\ref{prop:continuous-density-interface} gives a common doubling constant. Hence, by Theorem~\ref{sequential}, after passing to a subsequence, we have
\[
 \overline{S_1^c[U_j](0)}\longrightarrow K,
 \qquad
 \overline{S_{1/4}^c[U_j](0)}\longrightarrow\tfrac12K
\]
in Hausdorff distance. Since \(B_1\subset S_1^c[U_j](0)\subset B_{n^{3/2}}\), it follows from the above convergence that, for all large \(j\),
\[
 (1-\tau)\tfrac12S_1^c[U_j](0)
 \subset S_{1/4}^c[U_j](0)
 \subset(1+\tau)\tfrac12S_1^c[U_j](0).
\]
By \eqref{sectionrescaling}, this is precisely \eqref{goodKk} at \((x_j,k_j)\), contradicting \eqref{badblowup}. The claim follows.

Fix the resulting \(Q,\delta,k_*\). For a fixed center \(x_0\), abbreviate \(K_k=K_k(x_0)\) and \(\Delta_{k,Q}=\Delta_{k,Q}(x_0)\). By additivity,
\[
 \Delta_{k,Q}=\sum_{\ell=k-Q}^{k+Q-1}a_\ell.
\]
For \(N>2Q\), define
\[
 \mathcal B_N
 :=\left\{k\in\{Q,\ldots,N-Q\}:
               \Delta_{k,Q}>\delta\right\}.
\]
Since each \(a_\ell\) occurs in at most \(2Q\) of the sums \(\Delta_{k,Q}\), Proposition~\ref{prop:zero-average-final} gives
\[
 \#\mathcal B_N
 \le \frac{2Q}{\delta}\sum_{\ell=0}^{N-1}a_\ell
 =o(N)
\]
uniformly in \(x_0\).

Now take
\[
 k\in
 \{\max\{Q,k_*\},\ldots,N-Q\}\setminus\mathcal B_N.
\]
Then \(\Delta_{k,Q}\le\delta\), and the claim gives \eqref{goodKk}.

At every other scale, Lemma~\ref{fixratio}, applied with \(\theta=1/4\), gives
\begin{equation}\label{Kkstep}
 \frac{1}{4n^3}K_k
 \subset K_{k+1}
 \subset n^3K_k.
\end{equation}
Let \(\mathcal E_N\subset\{0,\ldots,N-1\}\) be the set of indices at which \eqref{goodKk} has not been obtained. The preceding argument shows that
\[
 b_N:=\#\mathcal E_N
 \le \#\mathcal B_N+2Q+k_*=o(N)
\]
uniformly in \(x_0\).

Iterating \eqref{goodKk} and \eqref{Kkstep}, with
\(g_N=N-b_N\), gives
\[
 2^{-g_N}(1-\tau)^{g_N}(4n^3)^{-b_N}K_0
 \subset K_N
 \subset
 2^{-g_N}(1+\tau)^{g_N}(n^3)^{b_N}K_0.
\]
Consequently, for a dimensional constant \(C\),
\begin{equation}\label{shapcontrol}
 2^{-N}e^{-C\tau N-Cb_N}K_0
 \subset K_N
 \subset
 2^{-N}e^{C\tau N+Cb_N}K_0.
\end{equation}
By Proposition~\ref{classicalproperties}, \(K_0\) contains and is contained in balls of uniform radii. Choose \(\tau\) sufficiently small and then \(N\) sufficiently large so that
\[
 C\tau N+Cb_N\le 2\varepsilon(\log2)N.
\]
Since
\[
 r_N^2=4^{-N}r_0^2,
\]
\eqref{shapcontrol} yields
\[
 B_{C_\varepsilon^{-1}(r_N^2)^{1/2+\varepsilon}}(x_0)
 \subset S_{r_N^2}^c[u](x_0)
 \subset
 B_{C_\varepsilon(r_N^2)^{1/2-\varepsilon}}(x_0)
\]
for all sufficiently large \(N\), uniformly in \(x_0\).

Finally, if \(r_{N+1}^2\le h\le r_N^2\), then \(h=\theta r_N^2\) for some \(\theta\in[1/4,1]\). Lemma~\ref{fixratio} transfers the preceding estimate from \(r_N^2\) to \(h\), after changing \(C_\varepsilon\). This proves \eqref{shapeestimate} for \(u\). The proof for \(v\) is identical after interchanging the source and target.
\end{proof}

\begin{proof}[Proof of Theorem \ref{mainthm}]
The global $W^{2,p}$ estimate follows from  Proposition \ref{prop:subpower-final}, Caffarelli's interior $W^{2,p}$ estimate \cite{Caffarelli1990W2p}, and a covering argument of Savin \cite{SavinGlobalW2p}.
\end{proof}

\appendix
\section{Local regularity}
\label{appendix}

The limiting transports arising in Sections~\ref{sec5limit} and Sections~\ref{sec6rigidity} may have infinite total mass. Since only bounded source and target windows are used in the regularity argument, the theorem of Figalli--Jhaveri admits the following locally finite formulation.

A Radon measure \(\lambda\) is \emph{locally doubling on ellipsoids} if, for every ball \(B\subset\mathbb R^m\), there is a constant \(D_B\ge 1\) such that
\begin{equation*}
 \lambda(E)\le D_B\lambda(\tfrac12E)
\end{equation*}
whenever \(E\subset B\) is an ellipsoid with center in \(\operatorname{spt}\lambda\). Here \(\tfrac12E\) is the concentric half of \(E\). We shall use without further comment the fact that local doubling on ellipsoids implies local doubling on bounded convex sets; see \cite[Corollary~2.5]{JhaveriSavin2022}.

\begin{proposition}
\label{prop:localfinite}
Let \(X,Y\subset\mathbb R^m\) be open convex sets and let
\[
 \mu=a\mathbf1_X\,\dd x,
 \qquad
 \nu=b\mathbf1_Y\,\dd y,
 \qquad a,b>0\quad\text{a.e.}
\]
be nonzero locally finite Radon measures. Assume that both are locally doubling on ellipsoids. Let \(w:X\to\mathbb R\) be a finite convex potential function with \((Dw)_\#\mu=\nu\), and let
\begin{equation*}
 \underline w(x)=
 \sup_{\substack{z\in X\\p\in\partial w(z)}}
 \{w(z)+p\cdot(x-z)\}
\end{equation*}
be its lower-semicontinuous minimal extension. Suppose that
\begin{equation*}
 \partial\underline w(\mathbb R^m)\subset\overline Y.
 \end{equation*}
Then \(\underline w\) is strictly convex and \(C^1\) in \(X\), and $Dw:X\longrightarrow Y$ is a homeomorphism.
\end{proposition}

\begin{proof}
The argument is a localization of the proofs of \cite[Theorems~1.2 and~1.6]{FigalliJhaveri2023}. We point out only the modifications needed for locally finite measures and possibly unbounded supports.

Extend \(\mu\) and \(\nu\) by zero outside \(X\) and \(Y\), respectively. We have $ \operatorname{spt}\mu=\overline X $ and $ \operatorname{spt}\nu=\overline Y.$  By the mass-balance formula \cite[Lemma~2.1]{FigalliJhaveri2023}, we know
\begin{equation*}
 \mu(E)=\nu\bigl(\partial\underline w(E)\bigr)
 \qquad\text{for every Borel set }E\subset\mathbb R^m.
\end{equation*}

Next, we first prove strict convexity. Let $\ell$ be an affine function supporting $w$ at some point $x_0\in X$. After subtracting \(\ell\), we may assume that $ w\ge0$ and $ w(x_0)=0.$ Define $\Sigma:=\{\underline w=0\}.$ Note that $\Sigma$ is closed and convex. We prove that it is a singleton by contradiction.

If \(\Sigma\) contains a line, then \(\partial\underline w(\mathbb R^m)\) is contained in a proper affine hyperplane. This contradicts
\((D\underline w)_\#\mu=\nu\), since \(b>0\) almost everywhere in the open set \(Y\). Hence \(\Sigma\) has an exposed point.

Suppose first that there is an exposed point in $\overline X\cap\operatorname{int}(\operatorname{dom}\underline w).$ After a translation, we may assume that the exposed point is the origin. Similarly as in \cite[ Theorem~1.2, Case~2a]{FigalliJhaveri2023}, choose a bounded open set \(U\Subset\operatorname{int}(\operatorname{dom}\underline w)\) containing the origin and a ball compactly contained in \(X\). Define
\[
 \Upsilon
 :=
 \operatorname{conv}\overline{\{D\underline w(x):x\in U,\  \underline w\text{ is differentiable at }x\}}.
\]
Then \(\Upsilon\) is compact and convex, and \(\partial\underline w(U)\subset\Upsilon\). Define
\[
 \Omega_U=\partial\underline w^*(\Upsilon),\qquad
 \rho=\mu\lfloor\Omega_U,\qquad
 \gamma=\nu\lfloor\Upsilon,
\]
and introduce the truncated minimal extension
\begin{equation*}
 \phi
 =\sup_{p\in\Upsilon}
   \{p\cdot x-\underline w^*(p)\}.
\end{equation*}
Then \(\phi\) is globally finite and Lipschitz,
\[
 \phi=\underline w\quad\text{on }\Omega_U,\qquad
 \partial\phi(\mathbb R^m)\subset\Upsilon,
 \qquad
 (D\phi)_\#\rho=\gamma.
\]
Moreover, $0<\rho(\mathbb R^m)=\gamma(\mathbb R^m)<\infty.$ Since \(\phi=\underline w\) near the origin and \(0\le\phi\le\underline w\), the origin remains an exposed point of $\{\phi=0\}$. We may now apply the John-normalized long-cone argument in \cite[proof of Theorem~1.2, Case~2a]{FigalliJhaveri2023}. All source and target sets occurring there are contained, in the original variables, in fixed bounded sets. Thus only the corresponding local doubling constants of \(\mu\) and \(\nu\) are used, and the same packing contradiction follows.

Suppose next that the exposed point belongs to
$
 \overline X\cap \partial(\operatorname{dom}\underline w).
$
After the normalization used in \cite[proof of Theorem~1.2, Case~2b]{FigalliJhaveri2023}, let
\[
 \underline w_\varepsilon(x)
 =\underline w(x)-\varepsilon(x_1+1),\qquad
 S_\varepsilon=\{\underline w_\varepsilon\le0\},
\]
and let \(A_\varepsilon\) be the John map of \(S_\varepsilon\), with linear part \(L_\varepsilon\). If \(q\) is an exterior normal to
\(\operatorname{dom}\underline w\) at the exposed point, the same polar construction implies
\[
 \operatorname{conv}\bigl(
 B_r\cup(a_\varepsilon+\mathbb R_+\xi_\varepsilon)\bigr)
 \subset
 \partial\widetilde w_\varepsilon
          (\widetilde S_\varepsilon),
\]
where
\[
 a_\varepsilon=-L_\varepsilon^{-t}e_1,\qquad
 \xi_\varepsilon=
 \frac{L_\varepsilon^{-t}q}{|L_\varepsilon^{-t}q|}.
\]
Choose a finite truncated cone whose normalized length tends to infinity, and the images of all target sets used in the packing argument remain in fixed bounded sets in the original target variables. The disjoint long-cone packing argument of \cite[(2.3)--(2.7)]{FigalliJhaveri2023}, using the local source and target doubling constants on those fixed bounded sets, again yields a contradiction.

Finally, suppose that the exposed point does not belong to \(\overline X\). We may assume that the exposed point is the origin and that
\[
 S_0:=\Sigma\cap\{x_1\ge-1\}
 \Subset\mathbb R^m\setminus\overline X.
\]
For
\[
 \underline w_\varepsilon
 =\underline w-\varepsilon(x_1+1),\qquad
 S_\varepsilon=\{\underline w_\varepsilon\le0\},
\]
one has \(S_\varepsilon\Subset\mathbb R^m\setminus\overline X\) for all small \(\varepsilon>0\). The polar estimate from \cite[proof of Theorem~1.2, Case~3]{FigalliJhaveri2023} implies that
\[
 B_{c\varepsilon}
 \subset\partial\underline w_\varepsilon(S_\varepsilon)
 \subset\overline Y-\varepsilon e_1.
\]
Consequently, with
\(\nu_\varepsilon=(\operatorname{Id}-\varepsilon e_1)_\#\nu\),
\[
 0=\mu(S_\varepsilon)
  =\nu_\varepsilon\bigl(
       \partial\underline w_\varepsilon(S_\varepsilon)\bigr)
  \ge\nu_\varepsilon(B_{c\varepsilon})>0,
\]
a contradiction. Thus \(\underline w\) is strictly convex in \(X\).

The differentiability argument is the localized version of \cite[proof of Theorem~1.6, Part~1]{FigalliJhaveri2023}. Fix $x_0\in X$. Suppose that $\partial \underline w(x_0)$ is not a singleton. Since $x_0\in\operatorname{int}(\operatorname{dom}\underline w)$, $\partial \underline w(x_0)$ is compact. Choosing an exposed point of this compact convex set and applying the affine normalization used in \cite[proof of Theorem~1.6, Part~1]{FigalliJhaveri2023}, \(\underline w\) is differentiable in \(X\). 

It remains to identify the range. Let $v=\underline w^*|_Y$ and let \(\underline v\) be its minimal extension. Then $v$ is a finite convex function on $Y$ and satisfies
\[
 (Dv)_\#\nu=\mu,\qquad
 \partial\underline v(\mathbb R^m)\subset\overline X.
\]
Applying the preceding argument with source and target interchanged shows that \(\underline v\) is strictly convex and \(C^1\) in \(Y\). Thus $T=D\underline w|_X,\ S=D\underline v|_Y$ are continuous and injective. By invariance of domain, we have
\[
 T(X)\subset Y,\qquad S(Y)\subset X.
\]
Moreover, by Legendre duality, we have
\[
 S\circ T=\operatorname{id}_X,\qquad
 T\circ S=\operatorname{id}_Y.
\]
Therefore \(T(X)=Y\), and $D\underline w:X\longrightarrow Y$ is a homeomorphism.
\end{proof}

\end{document}